\pdfoutput=1

\documentclass{amsart}
\usepackage[square, numbers]{natbib}

\usepackage{graphicx, amsmath, tikz-cd, array, amssymb,  enumitem, multirow, float, comment, xcolor, dsfont}

\usepackage{amsthm, enumitem}

\usepackage[colorlinks = true, citecolor = blue, urlcolor = blue]{hyperref}

\usepackage[section]{placeins}

\usepackage{caption, subcaption}

\usepackage{mathabx}

\newtheorem{theorem}{Theorem}[section]
\newtheorem*{theorem*}{Theorem}
\newtheorem{lemma}[theorem]{Lemma}
\newtheorem{corollary}[theorem]{Corollary}
\newtheorem{prop}[theorem]{Proposition}

\theoremstyle{definition}
\newtheorem{definition}[theorem]{Definition}
\newtheorem{example}[theorem]{Example}

\theoremstyle{remark}
\newtheorem{remark}[theorem]{Remark}

\newcommand{\PH}[1]{\mathrm{{PH}}}

\def\Z{{\mathbb Z}}
\def\N{{\mathbb N}}
\def\R{{\mathbb R}}
\def\C{{\mathbb{C}}}

\def\C{\mathbb{C}}
\def\F{\mathbb{F}}

\def\N{\mathbb{N}}
\def\Q{\mathbb{Q}}
\def\R{\mathbb{R}}
\def\T{\mathbb{T}}
\def\X{\mathbb{X}}

\def\Z{\mathbb{Z}}
\def\SW{\mathbb{SW}}

\def\aa{\mathbf{a}}
\def\dd{\mathbf{d}}
\def\kk{\mathbf{k}}
\def\tt{\mathbf{t}}
\def\uu{\mathbf{u}}
\def\vv{\mathbf{v}}

\def\xx{\mathbf{x}}

\def\zz{\mathbf{z}}

\def\bcd{\mathsf{bcd}}
\def\dgm{\mathsf{dgm}}

\def\<{\langle}
\def\>{\rangle}
\def\({\left(}
\def\){\right)}

\DeclareMathOperator*{\argmax}{\arg\!\max\;}

\definecolor{turquoise}{rgb}{0.2, 0.93, 0.75}

\begin{document}
	
	\title{Sliding window persistence of quasiperiodic functions}

	\author[Hitesh Gakhar]{Hitesh Gakhar}
	\address{
		\shortstack[l]{
			Department of Mathematics, \\
			Michigan State University \\
			East Lansing, MI, USA.}}
	\email{gakharhi@msu.edu}
	
	\author[Jose Perea]{Jose A. Perea }
	\address{
		\shortstack[l]{
			Department of Mathematics, and\\
			Khoury College of Computer Sciences,\\
			Northeastern University\\
			Boston, MA, USA.}}
	\email{j.pereabenitez@northeastern.edu}
	\thanks{\\[-.2cm] This work was partially supported by the National Science Foundation through grants DMS-1622301,  CCF-2006661,
and CAREER award  DMS-1943758. \\[.2cm]
Preliminary versions of these results  appeared in the first author's PhD thesis \cite{gakhar2020topological}.
	}

	
	
	
	
	\date{}
	
	\dedicatory{}
	
	
	\begin{abstract}
		A function is called quasiperiodic if its fundamental frequencies  are linearly independent over the rationals.
		With appropriate parameters,
		the sliding window point clouds of such  functions can be shown to be dense in  tori with dimension equal to the number of independent frequencies.
		In this paper, we develop  theoretical and computational techniques
		to study  the persistent homology of such sets.
		Specifically,
		we  provide   parameter optimization schemes  for sliding windows of quasiperiodic functions,
		and present theoretical lower bounds on their Rips persistent homology.
		The latter   leverages
		a recent persistent K\"{u}nneth formula.
		The theory is illustrated via computational examples  and an application to dissonance
		detection  in music audio samples.
	\end{abstract}
	
	\maketitle
	\section{Introduction}\label{section:Intro}
Recurrent behavior---both in time and space---is ubiquitous in nature.
\emph{Periodicity} and \emph{quasiperiodicity} are two prominent examples,
characterized by a vector of underlying non-zero frequencies:  If all pairwise   ratios are rational, then the recurrence is periodic, while
quasiperiodicity, on the other hand, occurs if there are at least two frequencies whose quotient is  irrational. Quasiperiodic recurrence  is at the heart of KAM (Kolmogorov-Arnold-Moser) theory \cite{broer2004kam},
it appears  as a signature of biphonation (i.e., the voicing of two simultaneous pitches) in mammalian vocalization \cite{mammalsubharmonics},
in climate change patterns on Mars \cite{pollack1982quasi},
in the oscillatory movement of the star TT Arietis \cite{hollander1992quasi},
and in the brain functioning in mice as reported by fMRI scans \cite{belloy2017dynamic}.
The list goes on.

Quasiperiodicity in dynamical systems is  typically  studied with numerical methods including Birkhoff averages \cite{das2016measuring}, periodic approximations \cite{slater1967gaps,sos1958distribution}, estimation of Lyapunov exponents \cite{weixing1993quasiperiodic}, power spectra \cite{wojewoda1993complex}, and recurrence quantification analysis \cite{webber1994dynamical,zbilut2002recurrence}.
New techniques from  applied topology have emerged recently as  complements to these
traditional   approaches in the task of recurrence detection---specifically for periodicity and quasiperiodicity quantification---in time series data \cite{sw1pers,toroidal,tralie2018quasi}.
This novel framework combines two key ingredients: \emph{sliding window embeddings} and \emph{persistent homology}.

Sliding window (also known as time-delay) embeddings
provide a framework to reconstruct the topology of state-space  attractors in dynamical systems, given observed time series data.
Indeed, given parameters $d \in \N = \{0,1, \ldots \}$ (controlling the embedding dimension $d+1$) and $\tau \in  \{x\in \R \mid x > 0\}$ (the time delay)
the  sliding window embedding of   $f: \R \longrightarrow \C$ at $t\in \R$  is the vector
\begin{align}
	SW_{d,\tau}f(t) :=
	\begin{bmatrix}
		f(t) \\ f(t+\tau) \\ \vdots \\ f(t+d\tau)
	\end{bmatrix} \in \C^{d+1}.
\end{align}
The motivation behind this construction is  Takens' embedding theorem \cite{takens1981detecting},
which asserts that   if $f$ is the result of observing the evolution of a
(potentially unknown) dynamical system,
then
the underlying topology  of the sliding window  point cloud $\SW_{d, \tau}f:=SW_{d,\tau}f(\R)$---generically in $f$ and for appropriate parameters $d,\tau$---recovers that of the traversed portion of the state space.
In particular,  this is how attractors can be reconstructed from
observed time series data.

The topology of attractors constrains many properties of the underlying dynamical system (e.g., periodic orbits, chaos, etc)
and detecting these features in practice is where persistent homology has come into play \cite{robins1999towards}.
Persistent homology is a  tool from Topological Data Analysis
widely used to quantify multiscale homological features of shapes.
Its typical input  is a collection $\mathcal{K} = \{ K_\epsilon \}_{\epsilon \geq 0}$ of  spaces
with  $K_\epsilon \subset K_{\epsilon'}$ continuous for all $\epsilon \leq \epsilon'$.
This is called a \emph{filtration}.
The output in each dimension $j \in \N$ is
a multiset
\[
\dgm_j(\mathcal{K})\subset \{(x,y) \in [0,\infty]\times [0,\infty] \mid  0 \leq  x < y \}\] called the $j$-th \emph{persistence diagram} of $\mathcal{K}$, where
each pair $(a, b) \in \dgm_j(\mathcal{K})$ encodes a $j$-dimensional topological feature
(like a connected component, a hole, or a void) which appears at $K_a$  and disappears entering  $K_b$.
The quantity  $b - a$  is the \emph{persistence} of the feature, and typically measures significance across the filtration.

In data analysis applications the input to persistent homology  is  often a  metric space $(X,\dd_X)$---e.g., a sliding window point cloud $\SW_{d,\tau} f$---from which  the \emph{Rips (simplicial) complex}
\begin{equation}\label{eqn: RipsCmplx}
	R_\epsilon(X,\dd_X) := \left\{ \{x_0, \ldots, x_n\} \subset X \mid  \max_{0 \leq  j,k \leq n} \dd_X(x_j , x_k) < \epsilon
	\;\; ,\;\;
	n\in \N
	\right\}
\end{equation}
is computed at each scale $\epsilon \geq 0$, producing  the \emph{Rips filtration}
\begin{equation}\label{eqn: RipsFilt}
	\mathcal{R}(X,\dd_X) := \{R_{\epsilon}(X,\dd_X)\}_{\epsilon \geq 0}.
\end{equation}
Points in the   Rips persistence diagrams $\dgm_j^\mathcal{R}(X) := \dgm_j\big(\mathcal{R}(X,\dd_X)\big)$ quantify the underlying topology
of $X$ in that pairs $(a,b)$ with large persistence $b-a$ represent likely topological features of a continuous space
around which $X$ accumulates.

The  diagrams  $\dgm_j^\mathcal{R}(\SW_{d,\tau} f)$ have   shown to be  rich signatures for recurrence detection in time series,
with applications including:
periodicity quantification in gene expression  data \cite{perea2015sw1pers},
(quasi)periodicity detection in videos \cite{tralie2018quasi},
synthesis of slow-motion videos from repetitive movements \cite{tralie2018slowmotion},
wheezing detection  \cite{emrani2014persistent},
and  chatter prediction  \cite{khasawneh2018chatter}.
See  \cite{perea2019notices} for a recent  survey.
One of the main challenges   in these applications  is  the validation of   empirical results, which
stems, in part, from the current limited
theoretical understanding  of how $\dgm_j^\mathcal{R}(\SW_{d,\tau} f)$ depends on $f,d,\tau$ and $T$.
That said, there are recent explicit conditions on $f$ for $\SW_{d,\tau} f$ to
provide appropriate reconstructions \cite{xu2019twisty},
as well as
analyses of sliding window persistence for  periodic functions \cite{sw1pers}, and quasiperiodic functions of the form \cite{toroidal}
\begin{equation}\label{eqn: QuasiExample}
	f(t) = c_1 e^{it\omega_1} + \cdots + c_N e^{it \omega_N}.
\end{equation}
In Eq. (\ref{eqn: QuasiExample}) the  $\omega_n > 0$ are $\Q$-linearly independent (i.e., incommensurate), and the coefficients $c_n \in \C$  are nonzero.
Our goal in this paper   is to   extend \cite{sw1pers} and \cite{toroidal} to general quasiperiodic functions; i.e., those beyond  Eq. (\ref{eqn: QuasiExample}).

\subsection{Contributions}
The first contribution of this paper is methodological:
we develop techniques to study the persistent homology of sliding window point clouds from general quasiperiodic functions.
Specifically,
we show that if $f: \R \longrightarrow \C$ is  quasiperiodic  with
incommensurate frequencies
$\omega = (\omega_1,\ldots, \omega_N)$ (Definition \ref{def:quasiperiodic}), and if for $\kk \in \Z^N$, $K \in \N$,
we let
\[
\widehat{F}(\kk) = \lim_{\lambda \to \infty}
\frac{1}{\lambda}
\int_{0}^{\lambda} f(t) e^{-i\<\kk , t\omega \>} dt
\;\;\;\;\; , \;\;\;\;\;
S_K f(t) = \sum_{\|\kk\|_\infty \leq K } \widehat{F}(\kk)e^{i\<\kk, t\omega\>}
\]
then the Rips persistence diagrams $\dgm^\mathcal{R}_j(\SW_{d,\tau} f)$, $j\in \N$,
can be approximated   in bottleneck distance
by  $\dgm_j^\mathcal{R}
(\SW_{d,\tau} S_K f)$ as $K\to \infty$.
The   diagrams of  $\SW_{d,\tau}S_K f$ are then  studied directly with methods extending those of \cite{toroidal, sw1pers};  the approximation to $\dgm^\mathcal{R}_j(\SW_{d,\tau} f)$ is of order
$O\left( K^{ \frac{N}{2} - r } \right) $
when $\vert\widehat{F}(\kk)\vert = O\left( \|\kk\|_2^{-r}\right)$ and $r > N/2$
(Corollary \ref{Corollary: DGMapproximation}).

This approximation strategy leads to our second contribution: computational schemes
for optimizing the choice of parameters $d \in \N$ and $\tau> 0$, so that the geometry (and hence the Rips persistent homology) of the sliding window point cloud
$\SW_{d,\tau} f = SW_{d,\tau} f(\R)$ robustly reflects that of
an $N$-torus. Explicitly:
\\

\fbox{
	\parbox{0.9\textwidth}{
		\begin{enumerate}[leftmargin=*]
			\item Given $f$, estimate  the coefficients $\widehat{F}(\kk)$
			and their frequency locations $\<\kk,\omega\>$. This can be done numerically
			with methods leveraging   the Discrete Fourier Transform \cite{gomez2010collocation, laskar1993frequency} or Wavelet analysis \cite{vela2002time}.
			\item Let  $K \in \N$  be the smallest integer so that
			\[
			\mathsf{supp}(\widehat{F}_K) :=
			\left\{
			\kk \in \Z^N  \mid \widehat{F}(\kk) \neq 0 \mbox{ and } \|\kk\|_\infty \leq K
			\right\}
			\]
			spans an $N$-dimensional $\Q$-vector space.
			This is possible provided $f$ is smooth enough (Lemma \ref{lemma: lattice}).
			\item Let $d$ be the cardinality of $\mathsf{supp}(\widehat{F}_K)$, or alternatively, the number of prominent peaks in the spectrum of $f$. This choice   is so that $\SW_{d,\tau} f$ has the right toroidal dimension (Theorem \ref{thm: Structure}).
			\item Let $\tau >0 $ be a   minimizer over $[0, \tau_\mathsf{max}]$ of the scalar function
			\[
			\Gamma(x) = \sum_{\kk \neq \kk'} \left\vert 1 +  e^{i x \langle \kk - \kk', \omega\rangle} + \cdots +  e^{i x \langle \kk - \kk', \omega\rangle d} \right\vert^2
			\] where the sum runs over $\kk,\kk' \in \mathsf{supp}(\widehat{F}_K)$. This choice is meant to amplify the toroidal features in $\dgm^\mathcal{R}_j(\SW_{d,\tau} f)$---see Figures \ref{fig:dgmsDiffTaus} and \ref{fig:PersVsTau}---and can be implemented via simple minimization algorithms.
		\end{enumerate}
	}
}
\\[.2cm]
Here, by toroidal features, we refer to the torus shaped attractors in the underlying dynamical system, which are captured by computing the persistent homology of the sliding window point cloud. By strong $N$-toroidal features, we mean that there are $N \choose j$ number of significant persistence points in $\dgm^\mathcal{R}_j$. 

Our third contribution leverages the aforementioned
approximation strategy,  parameter choices, and a recent persistent K\"unneth formula \cite{Kunneth},
to establish  bounds
for the cardinality and persistence of strong toroidal features
in $\dgm_j^\mathcal{R}(\SW_{d,\tau} f)$, for $1\leq j \leq N$.
We prove the following (Section \ref{section: persistent homology}):

\begin{theorem*}
	With $K,d\in \N$ and $\tau> 0$ as before, let $\sigma_{\mathsf{min}} > 0$ be the smallest
	singular value of the $(d+1)\times d$ Vandermonde matrix
	$\left[e^{i\<\kk,\omega\>\tau j}\right]_{ j =0,\ldots, d, \; \kk \in \mathsf{supp}(\widehat{F}_K)}$.
	Moreover,
	let $\kk_1,\ldots, \kk_N \in \mathsf{supp}(\widehat{F}_K)$ be $\Q$-linearly independent with
	\[
	\vert\widehat{F}(\kk_1)\vert \geq \vert\widehat{F}(\kk_2)\vert\geq \cdots \geq
	\vert\widehat{F}(\kk_N)\vert > 0.
	\]
	For  $1\leq n \leq N$, let $1\leq n_1 < \cdots < n_\ell \leq N$ be the longest sequence (i.e., largest $\ell$) for which
	\[
	\vert\widehat{F}(\kk_n)\vert =  \vert\widehat{F}(\kk_{n_1})\vert = \cdots = \vert\widehat{F}(\kk_{n_\ell})\vert
	\]
	and for  $1\leq j \leq N$, let
	\[
	\mu_j(n) :=
	{n_1  -1 \choose j - 1} + \cdots + {n_\ell -1  \choose j-1}.
	\]
	Then, there are at least $\mu_j(1) +\cdots +\mu_j(n)$
	toroidal features
	$(a,b) \in \dgm^\mathcal{R}_j\left( \SW_{d,\tau} f \right)$, counted with multiplicity, and
	with persistence
	\begin{equation}\label{eqn: LowerBound}
		b-a \;\;\geq \;\; \sqrt{3}\vert\widehat{F}(\kk_n)\vert\sigma_{\mathsf{min}}
		\;-\;
		4\sqrt{d+1}\|f - S_K f\|_\infty.
	\end{equation}
	
\end{theorem*}
We envision for these kinds of theorems to be useful in separating
noise from features in applications of  quasiperiodicity detection/quantification
with sliding windows and persistence.
As an illustration, consider the quasiperiodic signal
\[
f(t) =  2\sin(t) + 1.8\sin\left(\sqrt{3}t\right)
\]
shown in Figure \ref{fig:BoundFig} below (left), together with $\dgm_j^\mathcal{R}(\SW_{d,\tau} f)$ (right) in dimensions $j=1$ (blue) and $j=2$ (orange), computed with appropriate parameters $d,\tau$. The corresponding  theoretical lower bounds on persistence from Eq. (\ref{eqn: LowerBound}) are  depicted with dashed lines. We also use $f(t)$ to demonstrate the effect of random noise on sliding window persistence. See Section \ref{subsection: example3freq} for computational details.
\begin{figure}[!htb]
	\centering
	\includegraphics[width=\textwidth]{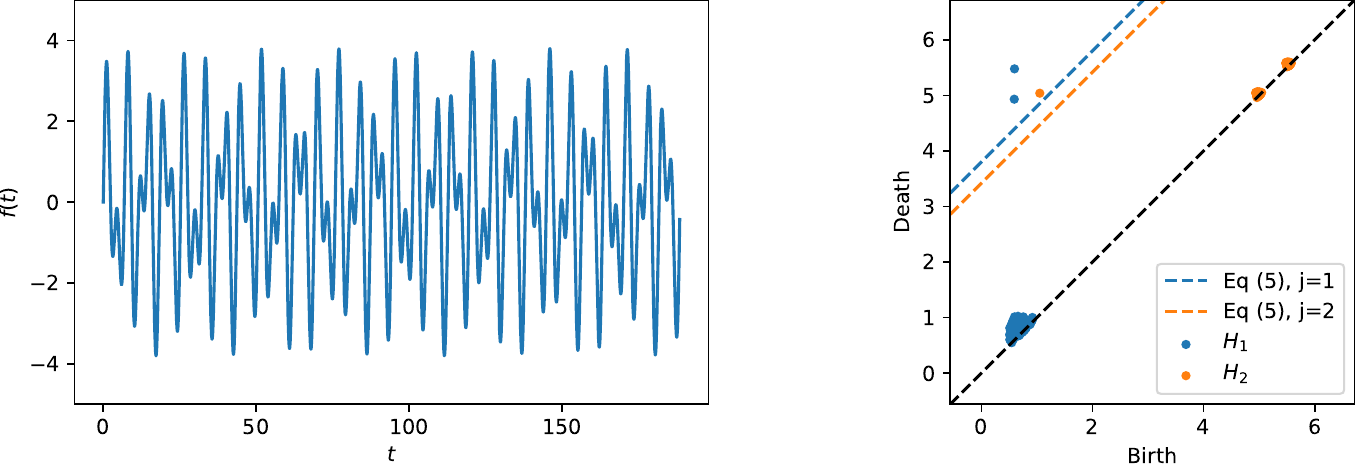}
	\caption{Left: the function $f(t) = 2\sin(t) + 1.8\sin\left(\sqrt{3}t\right)$. Right: The persistence diagrams in dimensions $j=1$ (blue) and $j=2$ (orange) of the sliding window point cloud $\SW_{d,\tau} f$, together with the lower bounds on persistence (dashed lines) from Eq. (\ref{eqn: LowerBound}).}
	\label{fig:BoundFig}
\end{figure}

Finally, as an application of the theory established in this paper,
we illustrate how sliding windows and persistence can be used to
identify the presence of dissonance in music audio recordings.
Indeed, dissonant intervals in music (like \emph{the tritone}) lead to quasiperiodic
recurrence in the recorded audio waves, which we show can be effectively detected with the methods developed here. We also note that the preliminary versions of these results  appeared in the first author's Ph.D. thesis \cite{gakhar2020topological}.

\subsection{Organization}
Section \ref{sec:background} presents the necessary mathematical background  on the Fourier theory of quasiperiodic functions and persistent homology.
Section \ref{section:fourierseries}  provides Fourier-type approximation theorems with explicit rates, at the level of  sliding window point clouds and Rips persistent homology.
Section \ref{section: structuretheorem} is devoted
to studying the geometric structure of the sliding window point cloud $\SW_{d,\tau} S_K f$ as a function of the parameters involved.
In Section \ref{section: parameterselection}, we establish a computational framework for the optimization of $d$ and $\tau$.
Section \ref{section: persistent homology} establishes the advertised theoretical lower bounds on $\dgm_j^{\mathcal{R}}(\SW_{d,\tau} f)$, and
we end in Section \ref{section: applications} with computational examples and applications.

\subsection{Notation}\label{notation}
Let $\T = \mathbb{R}/2\pi\mathbb{Z} \cong S^1$, that is, $[0,2\pi]$ with the endpoints identified.
Similarly for $N \in \N$, let $\T^N = \left(\mathbb{R}/2\pi\mathbb{Z}\right)^N \cong  S^1 \times \cdots \times S^1$.
We will abuse notation and regard any $F: \T^N \longrightarrow \C$ as both a function of a variable  $\tt \in \R^N$ where each coordinate $t_n$ is $2\pi$-periodic, and also as a function on the quotient $\T^N$.

\section{Preliminaries}\label{sec:background}
In this section we establish the necessary background for later parts of the paper.
In particular, we provide  a short review of Kronecker's multidimensional approximation theorem in Section \ref{subsection:kronecker}, as well as of quasiperiodic functions and their Fourier theory in  Section \ref{section:quasiperiodic}.
The theory of persistent homology is briefly discussed in Section \ref{section: Background persHom}.

\subsection{Kronecker's theorem}\label{subsection:kronecker}
If we regard $\R$  as a  vector space over $\Q$, then
a finite set $\lbrace \beta_1 , \dots, \beta_N \rbrace \subset \mathbb{R} $   is called \textit{incommensurate} if $\beta_1 , \dots,  \beta_N $ are linearly independent over $\Q$. That is, for $c_1, \dots, c_N \in \Q$
\[
c_1 \beta_1 + \cdots + c_N \beta_{N} = 0 \;\;\; \mbox{ if and only if  } \;\;\; c_1 = \cdots = c_N = 0.
\]
Kronecker's theorem is concerned with simultaneous diophantine approximations to real numbers, and incommensurability turns out to be the main condition.
Later on we will use this theorem to see why sliding window point clouds from   functions with incommensurate frequencies (i.e., quasiperiodic) are dense in high-dimensional tori.
For now, the theorem can be stated as follows \cite[Chapter 7]{apostol2012modular}.

\begin{theorem}[Kronecker]\label{Kroneckermulti}
	$\lbrace \beta_1 , \dots, \beta_N \rbrace \subset \R$ is incommensurate if and only if
	for every $r_1,\dots, r_N \in \R$ and every  $\epsilon > 0$,
	there exist $t \in \R$ and $k_1, \dots, k_N \in \Z$ so that
	\begin{equation}\label{eqn: kroneckerconditionbasic}
		\vert t\beta_n - r_n - 2\pi k_n \vert < \epsilon \hspace{1cm}\mbox{ for all }\hspace{1cm} 1 \leq n \leq N.
	\end{equation}	
	As a consequence, the entries of $ \beta = (\beta_1, \dots, \beta_{N}) \in \R^N$ are incommensurate  if and only if  \,
	$\R\beta / 2\pi \Z^N :=
	\big\{t(\beta_1,\ldots, \beta_N) \mod 2\pi \mid t\in \R
	\big\}$   is dense in  $ \T^N$.
\end{theorem}

\begin{remark}\label{rmk: IntegerKronecker}
	If one further assumes that $\{\pi ,\beta_1,\ldots, \beta_N\}$ is incommensurate,
	then Eq. (\ref{eqn: kroneckerconditionbasic}) holds with $t\in \Z$.
\end{remark}

Here is a useful consequence of Kronecker's theorem:
\begin{corollary}\label{Corollary: Kronecker}
	Let $\beta = (\beta_1, \dots, \beta_{\alpha})\in \R^\alpha$, for $\alpha \in \N$.
	Then $\mathsf{span}_\Q\{\beta_1,\ldots, \beta_\alpha\}$ has dimension $N$ over $\Q$,  if and only if
	\,$\R\beta/2\pi \Z^\alpha \subset \T^\alpha$ is dense in an $N$-torus
	embedded in $\mathbb{T}^{\alpha}$.
\end{corollary}

Before presenting the proof, and as an illustration of this Corollary,
Figure \ref{fig:Kronecker}   shows   what $T\beta/2\pi\Z^3$ looks like inside $\T^3$ for $T:= \left[-10^4 , 10^4\right]\cap  \Z $ and $\beta$ equals one of
\[
\left(\sqrt{2}, \sqrt{3}, 0\right), \; \left(\sqrt{2},\sqrt{3}, \sqrt{2} + \sqrt{3}\right), \; \left(\sqrt{2},\sqrt{3}, 3\sqrt{2} + 2\sqrt{3}\right), \; \left(\sqrt{2},\sqrt{3},\pi^2+1\right).
\]
\begin{figure}[!htb]
	\centering
	\hspace{-.2cm}
	\begin{subfigure}{0.24\textwidth}
		\centering
		\includegraphics[width = \textwidth]{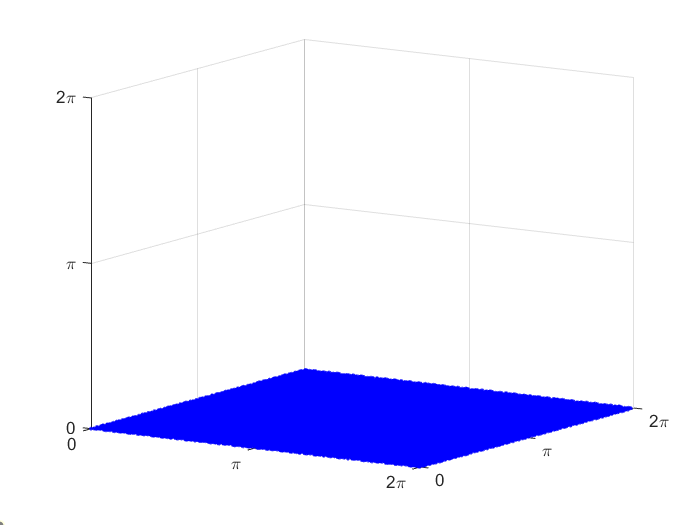}
	\end{subfigure}
	\begin{subfigure}{0.24\textwidth}
		\centering
		\includegraphics[width = \textwidth]{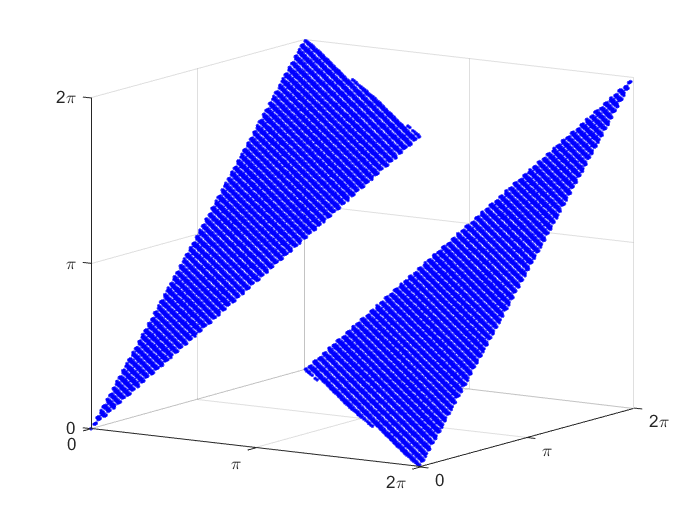}
	\end{subfigure}
	\begin{subfigure}{0.24\textwidth}
		\centering
		\includegraphics[width = \textwidth]{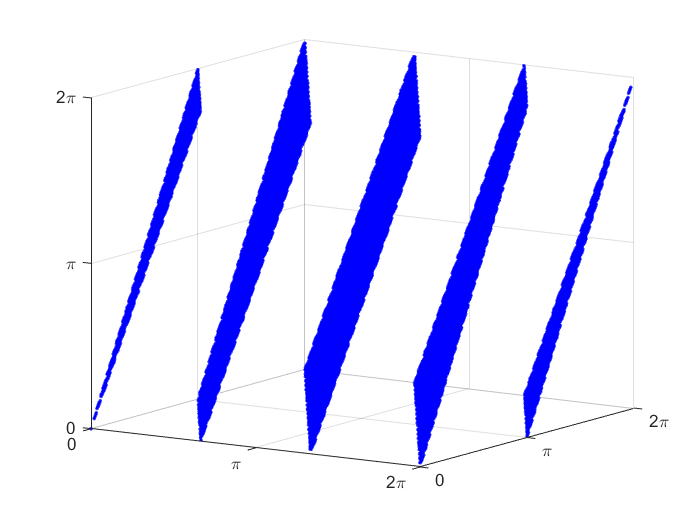}
	\end{subfigure}
	\begin{subfigure}{0.24\textwidth}
		\centering
		\includegraphics[width = \textwidth]{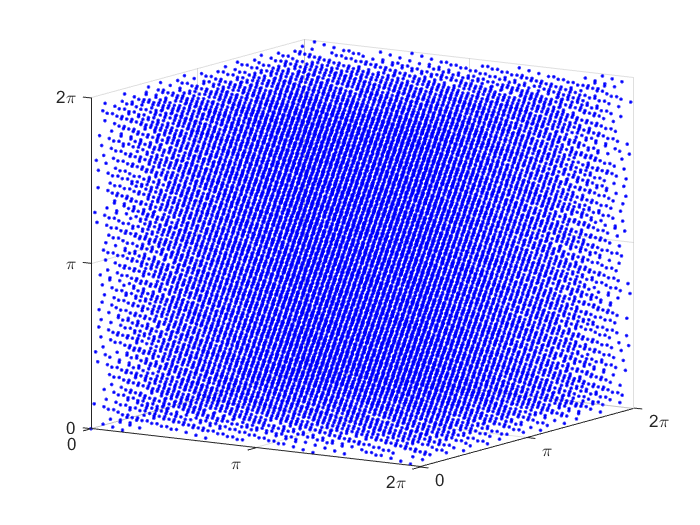}
	\end{subfigure}
	\caption
	{The set $T\beta/2\pi\Z^3 \subset \T^3$ for     $\beta=  (\sqrt{2},\sqrt{3}, 0)$ (left, a 2-torus), $\beta= (\sqrt{2},\sqrt{3}, \sqrt{2}+\sqrt{3})$ (second, a 2-torus), $ \beta = (\sqrt{2}, \sqrt{3} ,3\sqrt{2}+ 2\sqrt{3} )$ (third, a 2-torus), $\beta=(\sqrt{2},\sqrt{3}, \pi^2+1)$ (right, a 3-torus) and $T = \left[-10^4, 10^4\right]\cap \Z$.}
	\label{fig:Kronecker}
\end{figure}

\noindent Notice that in each case, $T\beta/2\pi \Z^3$ traces a torus of the appropriate dimension, but embedded in ways governed by the linear relations between the entries of $\beta$.
\begin{proof}[Proof (of Corollary \ref{Corollary: Kronecker})]
	Let $\lbrace \beta_1, \dots, \beta_{\alpha} \rbrace \subset \R$ span
	an $N$-dimensional vector space over $\mathbb{Q}$,
	and without loss of generality,
	assume that   $\beta_1, \dots, \beta_N$
	form a basis, $N < \alpha$.
	For $1 \leq j \leq \alpha -N$ let
	$k_j, k_{j, n} \in \Z$ be so that
	\begin{equation}\label{eqn: kbetas}
		k_j\beta_{N+j} = \sum_{n = 1}^N k_{j,n}  \beta_n
	\end{equation}
	If $N=1$ we further   require that $\mathsf{gcd}(k_j , k_{j,1}) = 1$, and if   $k = \mathsf{lcm}(k_1,\ldots, k_{\alpha -1})$,
	then it readily follows that
	\[
	\begin{array}{cccc}
		\T^1 & \longrightarrow & \T^\alpha  \\
		t & \mapsto & \frac{kt}{\beta_1}\beta
	\end{array}
	\]
	is an embedding of $\T^1$ into $\T^\alpha$ with the closure $\overline{\R\beta / 2\pi \Z^\alpha} =  \R\beta/2\pi \Z^\alpha$ as its image.
	
	Now, for $N\geq 2$ define $\phi:\R^N \longrightarrow \R^\alpha$ and $\psi : \R^\alpha \longrightarrow \R^\alpha$ as
	\begin{eqnarray}
		\phi(t_1,\ldots, t_N) &=&
		\left(t_1,\ldots, t_N, \sum\limits_{n = 1}^N k_{1,n} t_n \,, \ldots, \sum\limits_{n = 1}^N k_{\alpha - N,n}t_n\right) \\[.2cm]
		\psi(t_1,\ldots,t_\alpha) &=&
		\big(t_1, \ldots, t_N,\, k_1 t_{N+1}\,, \ldots, \,k_{\alpha -N} t_{\alpha}\big)
	\end{eqnarray}
	Note that both $\phi$ and $\psi$ preserve classes modulo $2\pi$, and therefore descend to continuous maps
	$\phi:\T^N \longrightarrow \T^\alpha$, $\psi : \T^\alpha \longrightarrow \T^\alpha$.
	Moreover, since   $\phi: \T^N \longrightarrow \T^\alpha$ is injective and tori are compact and Hausdorff, then $\phi$ is a homeomorphism onto its image.
	We claim that $\psi : \T^\alpha \longrightarrow \T^\alpha$ is injective when restricted to
	$\R\beta/2\pi\Z^\alpha$.
	Indeed, if $t\neq t'\in \R$ are so that $\psi(t\beta) = \psi(t'\beta)$ in $\T^\alpha$, then
	there exist $p,q\in \Z$ with
	\begin{eqnarray*}
		(t - t')\beta_1 &=& 2\pi p \\
		(t - t')\beta_2 &=& 2\pi q
	\end{eqnarray*}
	which implies $q\beta_1 = p\beta_2$, contradicting the $\Q$-linear independence of $\beta_1, \ldots,\beta_N$.
	It follows that $\psi$ is an injective continuous map on $\overline{\R\beta/2\pi\Z^\alpha}$, and  thus induces
	a homeomorphism
	$\overline{\R\beta/2\pi\Z^\alpha}
	\cong
	\psi\left(\overline{\R\beta/2\pi\Z^\alpha}\right)$.
	Finally, since  Eq. (\ref{eqn: kbetas})
	implies that   $\psi(t\beta) = \phi\big(t(\beta_1,\ldots, \beta_N)\big)$ for every $t\in \R$, and
	$\overline{\R(\beta_1,\ldots, \beta_N)/2\pi\Z^N}
	\cong \T^N $ by  Kronecker's theorem, then
	\[
	\overline{\R\beta/2\pi \Z^\alpha}
	\cong
	\psi\left(\overline{\R\beta/2\pi\Z^\alpha}\right)
	\cong
	\phi\left(\overline{\R(\beta_1,\ldots, \beta_N)/2\pi \Z^N}\right)
	=
	\phi\left(\T^N\right)
	\cong \T^N.
	\]
	The only if direction follows from the fact that two tori of different dimensions cannot be homeomorphic.
\end{proof}

\subsection{Quasiperiodic functions}\label{section:quasiperiodic}
As we alluded to in the introduction, quasiperiodic functions are
superpositions of periodic oscillators with incommensurate frequencies.
They arise in dynamical systems  as observation
functions on invariant toroidal submanifolds \cite{samoilenko2012elements}.
More specifically,
\begin{definition}\label{def:quasiperiodic}
	Let $\omega_1,\ldots, \omega_N > 0 $ be incommensurate.
	A function $f: \R \longrightarrow \C$ is said to be \emph{quasiperiodic} with frequency vector
	$\omega = (\omega_1,\ldots, \omega_N)$
	if
	\[
	f(t) = F(\omega_1 t, \ldots, \omega_N t)
	\]
	for all $t\in \R$ and  a  continuous function $F:\T^N\longrightarrow\C$.
	That is  $F\in C(\T^N )$,  which we will call a \emph{parent function} for $f$.
\end{definition}

\begin{remark}
	The case $N=1$ recovers the family of complex-valued periodic functions, and thus the results
	presented here generalize those of \cite{sw1pers}.
\end{remark}

\begin{remark}[\textbf{Important}]\label{rmk: MinmalFreqDim}
	We will require throughout that the dimension $N$ of the frequency vector $\omega \in \R^N$
	for $f$ quasiperiodic, be \emph{minimal}.
	The reason being that if not, then a function like $f(t) = e^{i(1 + \pi)t}$ can be regarded
	as being  quasiperiodic with frequency vector $\omega = 1 + \pi$ and parent function $F(t) = e^{it}$,
	or as having $\omega = (1,\pi)$ for frequency vector and $F(t_1,t_2) = e^{i(t_1 + t_2)}$ for parent function.
	Requiring that $N$ be minimal, and showing that for a given $\omega$ the parent function is unique
	(we will do so in Theorem \ref{thm: FourierQuasiCoeff} below), eliminates this ambiguity.
\end{remark}

It turns out that the traditional approximation theory via Fourier series on $L^2(\T^N)$ and $C(\T^N)$
can be leveraged to obtain similar insights for quasiperiodic functions.
We describe how in what follows (Theorem \ref{thm: UnifConvFourierQuasi}), though the interested reader should also consult
\cite{apostol2012modular,grafakos2008classical,samoilenko2012elements}.
Indeed, let  $\|(k_1,\ldots, k_N)\|_\infty = \max\limits_{1\leq n \leq N} \vert k_n \vert$,  and for $K\in \N$ let
$
I^N_K = \big\lbrace \kk  \in \Z^N \mid \|\kk \|_\infty \leq K \big\rbrace
$ be the integral square box of side $2K$.
The $K$-truncated Fourier polynomial of $F \in L^2\left(\T^N\right)$ is
the function
\begin{equation}\label{eqn: multifourier}
	S_KF(\tt ) = \sum_{\mathbf{k} \in I^N_K}\widehat{F}(\kk)e^{i\langle\kk,\tt \rangle}
\end{equation}
where $\tt  = (t_1, \dots, t_N) \in \R^N$, $\langle \cdot , \cdot \rangle$ is the standard inner product in $\mathbb{R}^N$, and
\begin{equation}\label{eqn: multifouriercoeffcients}
	\widehat{F}(\kk)
	\;\;=\;\;
	\frac{1}{(2\pi)^N} \int\limits^{2\pi}_{0}\cdots \int\limits^{2\pi}_{0}F(t_1,\dots,t_N) e^{-i \langle\kk,\tt\rangle} dt_1\cdots dt_N
	\;\;=\;\;
	\left\<F, e^{i\<\kk, \cdot \>}\right\>_{L^2}
\end{equation}
is the $\kk$-Fourier coefficient of $F$, for $\kk \in \Z^N$.
As it is well-known \cite[Proposition 3.2.7]{grafakos2008classical}, the sequence $\{S_K F\}_{K\in \N}$ converges to $F$  in $L^2(\T^N)$  as $K \to \infty$. That is,
\[
\lim_{K\to \infty} \|F - S_K F\|_{L^2} =0.
\]
It is not   the case, however, that one has pointwise convergence $S_K F(\tt) \to F(\tt)$, $\tt \in \T^N$, even for
$F\in C(\T^N)$ (see \cite[Proposition 3.4.6.]{grafakos2008classical} for negative results,
and Theorem \ref{thm: coefficientbounds} below
for appropriate conditions).

One can address these difficulties with \emph{approximate identities},
and in particular using the \emph{square Ces\`{a}ro (or Fej\'er) mean}
\begin{equation}\label{eqn: CesaroMean}
	C_K F (\tt) =
	\sum_{ \kk \in I_K^N} \left(1 - \frac{\vert k_1\vert}{K+1}\right)\cdots\left(1 - \frac{\vert k_N \vert}{K+1}\right)\widehat{F}(\kk)e^{i\<\kk,\tt\>}
\end{equation}
which for $F \in C(\T^N)$ satisfies
\begin{align}\label{eqn:Fejerstheorem}
	\lim_{K\to \infty} \|F - C_K F\|_\infty = 0
\end{align}
(this is Fej\'er's theorem) for $\|\cdot\|_\infty$  the sup-norm of uniform convergence in $C(\T^N)$.
We  now   state the first main result on the Fourier theory of quasiperiodic functions.
\begin{theorem}\label{thm: FourierQuasiCoeff}
	If $f: \R \longrightarrow \C$ is   quasiperiodic   with frequency vector $\omega \in \R^N$,
	then $f$ has a unique parent function $F \in C(\T^N)$ with  Fourier
	coefficients
	\begin{equation}\label{eqn: FourierQuasiCoeff}
		\widehat{F}(\kk) = \lim_{\lambda\to \infty} \frac{1}{\lambda} \int_{0}^\lambda f(t)e^{-i\<\kk, t\omega\>}dt.
	\end{equation}
\end{theorem}
\begin{proof}
	Write $f(t) = F(\omega t)$ for  $F\in C(\T^N)$.
	Since $\lim\limits_{K\to \infty}\|F - C_K F  \|_\infty =  0 $ (from Eq. (\ref{eqn:Fejerstheorem})), then   $C_K f(t) := C_KF(\omega t)$ converges
	to $f(t)$ uniformly in $t\in \R$.
	We claim that
	\[
	\frac{1}{\lambda} \int_{0}^\lambda C_K f(  t) e^{-i \<\kk , t \omega\>}dt
	\;\;\; \to \;\;\;
	\frac{1}{\lambda} \int_{0}^\lambda f(t) e^{-i \<\kk , t \omega\>} dt
	\]
	uniformly in $\lambda > 0$ as $K \to \infty$. Indeed,
	\[
	\left\vert\,
	\frac{1}{\lambda} \int_{0}^\lambda \big( C_K f( t)  - f(t) \big)e^{-i \<\kk , t \omega\>}dt
	\,\right\vert
	\;\leq \;\|C_K f  - f\|_\infty
	\]
	where the right hand side goes to zero as $K\to \infty$ independent of $\lambda$.
	Therefore, by the Moore-Osgood theorem,   we can exchange the order of limits as
	\begin{equation}\label{eqn: Integral}
		\lim_{\lambda\to \infty} \frac{1}{\lambda} \int_{0}^\lambda f(t)e^{-i\<\kk, t\omega\>}dt
		=
		\lim_{K \to \infty } \lim_{\lambda\to \infty} \frac{1}{\lambda} \int_{0}^\lambda C_Kf(t)e^{-i\<\kk, t\omega\>}dt
	\end{equation}
	and if
	\[
	\lambda_\kk = \left(1 - \frac{\vert k_1 \vert}{K+1}\right)\cdots\left(1 - \frac{\vert k_N \vert}{K+1}\right)\widehat{F}(\kk)
	\]
	are the coefficients of $C_KF$ (defined in Eq. (\ref{eqn: CesaroMean})), then
	evaluating  the right hand side of Eq. (\ref{eqn: Integral}) yields
	\begin{eqnarray*}
		\lim_{\lambda\to \infty} \frac{1}{\lambda} \int_{0}^\lambda f(t)e^{-i\<\kk, t\omega\>}dt
		&=&
		\lim_{K \to \infty } \lim_{\lambda\to \infty}
		\left(
		\lambda_\kk +
		\sum_{\kk' \in I_K^N \smallsetminus \{\kk\}}
		\lambda_{\kk'}\frac{e^{i\<\kk' - \kk, \lambda \omega\>} - 1 }{i\<\kk' - \kk, \omega\>\lambda}
		\right) \\[.2cm]
		&=& \lim_{K\to \infty} \left(1 - \frac{\vert k_1 \vert}{K+1}\right)\cdots\left(1 - \frac{\vert k_N \vert}{K+1}\right)\widehat{F}(\kk)\\[.2cm]
		&=& \widehat{F}(\kk).
	\end{eqnarray*}
	If there were another parent function $G\in C(\T^N)$---i.e., with $ f(t) = G(t\omega)$---then the above calculation shows that $\widehat{G}(\kk) = \widehat{F}(\kk)$ for every $\kk\in \Z^N$.
	Since functions in $L^2(\T^N)$ with the same Fourier coefficients are equal almost everywhere \cite[Proposition  3.2.7]{grafakos2008classical}, then
	continuity improves this to  functional equality   $G = F$.
\end{proof}

We now move onto providing conditions for the uniform convergence of
\begin{equation}\label{eqn: FourSeriesQuass}
	S_K f( t) := \sum_{\kk \in I_K^N} \widehat{F}(\kk)e^{i\<\kk,t\omega\>} \;\;\; , \;\;\; t\in \R \;\;\;, \;\;\; K \in \N
\end{equation}
as $K\to \infty$.
Here the $\widehat{F}(\kk)$ can be seen equivalently as the Fourier coefficients of the parent function $F$,
or as the result of evaluating  the right hand side of Eq. (\ref{eqn: FourierQuasiCoeff}).
The latter is what we expect to have access to in practice.
We start with an upper bound on the size of the coefficients $\widehat{F}(\mathbf{k})$    \cite[Theorem 3.3.9]{grafakos2008classical}.

\begin{prop}\label{prop: Fcoeffcientbound}
	Let $r \in \N$ and suppose that the partial derivatives $\partial^{\mathbf{l}}F$ exist and
	are continuous for all $\| \mathbf{l} \|_{1} = \vert l_1\vert  + \cdots + \vert l_N\vert  \leq r$.
	That is,   $F \in C^r(\T^N)$. Then
	\[
	\left\vert \widehat{F}(\mathbf{k}) \right\vert \leq \frac{\sqrt{N}^r}{ \| \mathbf{k} \|_{2}^r} \left\vert \widehat{\partial^r_nF}(\mathbf{k}) \right\vert
	\]
	where $n = n(\kk)$ satisfies
	$\vert k_n \vert = \| \mathbf{k} \|_{\infty}$ and
	${\partial^r_nF}$ is the $r$-th partial derivative of
	$F$ with respect to $t_n$.
\end{prop}

These types of inequalities can be used to estimate the degree $r$ of regularity of the parent function $F$,
by inspecting the rate of decay of the coefficients $\left\vert\widehat{F}(\kk)\right\vert$.
Proposition \ref{prop: Fcoeffcientbound}  yields the following estimate for uniform approximation error.

\begin{theorem}\label{thm: coefficientbounds}
	If  $F\in C^r(\T^N)$ for $ \frac{N}{2}  <r \in \N$,  then
	\begin{equation}\label{eqn:absconvergenceoftail}
		\sum_{\mathbf{k} \notin I^N_K}
		\left\vert\widehat{F}(\mathbf{k})  \right\vert
		\leq
		\left(
		\frac{Area(\mathbb{S}^{N-1}) N^r}{K^{2r- N}(2r-N)}
		\sum_{n=1}^N
		\big\| \partial^r_nF  -  S_K \partial^r_nF \big\|^2_{L^2}
		\right)^{1/2}
	\end{equation}	where
	${\partial^r_nF}$ is the $r$-th partial derivative of
	$F$ with respect to $t_n$.
	As a result, the sequence of Fourier coefficients $\widehat{F}(\kk)$ is absolutely summable, i.e.
	\[
	\sum_{\mathbf{k} \in \Z^N} \big\vert  \widehat{F}(\mathbf{k})  \big\vert < \infty.
	\]
\end{theorem}
\begin{proof}
	From Proposition \ref{prop: Fcoeffcientbound}  we have that
	\begin{equation}\label{eqn:coefficientbounds}
		\sum_{\mathbf{k} \notin I^N_K} \big\vert  \widehat{F}(\mathbf{k})  \big\vert \leq \sum_{\mathbf{k} \notin I^N_K}   \frac{\sqrt{N}^r }{\| \mathbf{k}\|_{2}^r }\left\vert \widehat{\partial^r_nF}(\mathbf{k}) \right\vert		\end{equation}
	for any fixed $K\in \N$.
	Note that $n = \argmax\limits_{1 \leq j \leq N} \vert k_j \vert$ depends on $\mathbf{k}$, so we will write it as $n(\kk)$, and
	the right hand side of Eq. (\ref{eqn:coefficientbounds}) can be bounded using Cauchy-Schwarz as
	\[
	\sum_{\kk \notin I^N_K}
	\frac{\sqrt{N}^r }{\| \kk \|^r_2}
	\left\vert \widehat{\partial^r_{n(\kk)}F}(\kk) \right\vert
	\leq
	\sqrt{N}^r
	\left( \sum_{\kk \notin I^N_K}  \frac{1}{\| \kk \|_{2}^{2r} } \right)^{1/2}
	\left( \sum_{\kk \notin I^N_K}\left\vert\widehat{\partial^r_{n(\kk)}F}(\kk) \right\vert^2 \right)^{1/2}.
	\]
	Moreover,   since $\partial^r_nF$ is continuous and thus  square integrable for each $n\in \N$, then
	its Fourier coefficients are square summable:
	\begin{equation}\label{eqn: squaresummable}
		\sum_{\mathbf{k} \in \Z^N}\left\vert\widehat{\partial^r_nF}(\mathbf{k}) \right\vert^2 = \left\|\partial^r_n F\right\|_{L^2}^2 < \infty
	\end{equation}
	by Parseval's theorem.		Hence, summing over  $n$ and  rearranging terms we get
	\[
	\sum_{\kk \notin I^N_K} \left\vert \widehat{\partial^r_{n(\kk)}F}(\kk) \right\vert^2
	\;\;\leq\;\;
	\sum_{n=1}^N \sum_{\kk \notin I^N_K} \left\vert \widehat{\partial^r_nF}(\kk) \right\vert^2
	\;\;= \;\;
	\sum_{n=1}^N
	\big\| \partial^r_nF  -  S_K \partial^r_nF \big\|^2_{L^2}
	\] which goes to zero as $K\to \infty$.
	In order to bound the remaining sum of fractions,
	let $J^N_K =
	\left\{\mathbf{y}  \in \mathbb{R}^N \; \mid \; \| \mathbf{y} \|_{\infty} \leq K\right\}
	$
	and  let $B^N_K =
	\left\{\mathbf{y}   \in \mathbb{R}^N \;  \mid \; \| \mathbf{y} \|_{2} \leq K\right\}
	$.
	Observe that
	\begin{align*}
		\sum_{\mathbf{k} \notin I^N_K}  \frac{1}{\| \mathbf{k} \|_{2}^{2r} }
		&\leq   \int\limits_{\mathbf{x} \notin J^N_K}  \frac{1}{ \left( y_1^2 + \cdots + y_N^2 \right)^{r}} dy_1 \cdots dy_N \\
		&\leq   \int\limits_{\mathbf{x} \notin B^N_K}  \frac{1}{ \left( y_1^2 + \cdots + y_N^2 \right)^{r}} dy_1 \cdots dy_N
	\end{align*}
	which in higher dimensional spherical coordinates can be written as
	\[
	\int\limits_{\mathbb{S}^{N-1}} \int\limits_{K}^{\infty}  \frac{1}{\rho^{2r - N + 1}}  d\rho d\sigma
	\]
	for $d\sigma$  the differential of surface area on the unit sphere $\mathbb{S}^{N-1}$ (the differential solid angle)
	and $\rho$  the distance from a point in $\R^N$ to the origin.
	The integral  satisfies
	\begin{eqnarray*}
		\int\limits_{\mathbb{S}^{N-1}} \int\limits_{K}^{\infty}  \frac{1}{\rho^{2r - N + 1}}  d\rho d\sigma
		&=&  Area\left(\mathbb{S}^{N-1}\right)  \frac{\rho^{N-2r}}{N-2r} \Bigg\vert^{\infty}_{K}  \\
		&=&  \frac{Area\left(\mathbb{S}^{N-1}\right)}{K^{2r - N}(2r-N)}
	\end{eqnarray*}
	and thus
	\[		
	\sum_{\mathbf{k} \notin I^N_K}
	\left\vert \widehat{F}(\mathbf{k})  \right\vert
	\leq
	\left(
	\frac{Area(\mathbb{S}^{N-1}) N^r}{K^{2r- N}(2r-N)}
	\sum_{n=1}^N \big\| \partial^r_nF  -  S_K \partial^r_nF \big\|^2_{L^2}
	\right)^{1/2}
	\]
	where the right hand side goes to zero as $K\rightarrow \infty$.
\end{proof}	

\begin{remark}See
	\cite[Chapter VII, Corollary 1.9]{stein2016introduction} for a result  akin to Theorem \ref{thm: coefficientbounds}.
	While both   have   similar hypotheses and deal with absolute Fourier convergence, Theorem \ref{thm: coefficientbounds} above gives explicit bounds for the  size of the error term  $\sum\limits_{\mathbf{k} \notin I^N_Z} \left\vert\widehat{F}(\kk)\right\vert$.
	We will need such explicit estimates when discussing Fourier approximations to  persistent homology
	of sliding window point clouds in Section \ref{section:fourierseries}.
\end{remark}

Now,  absolute summability of the Fourier coefficients $\widehat{F}(\kk)$ implies
uniform convergence $S_K F \rightarrow F$,
\cite[Chapter VII, Corollary 1.8]{stein2016introduction}, since
\[
\left\vert F(\tt) - S_K F(\tt)\right\vert \leq \sum_{\kk \notin I_K^N} \left\vert\widehat{F}(\kk)\right\vert
\]
for all $\tt \in \T^N$.	Combining  this fact with  Eq. (\ref{eqn:absconvergenceoftail}),
yields the following Fourier series approximation theorem for quasiperiodic functions:
\begin{theorem}\label{thm: UnifConvFourierQuasi}
	Let $f$ be quasiperiodic with parent function $F \in C^r(\T^N)$, $r > \frac{N}{2}$.
	If $S_K f$ is defined as in Eq. (\ref{eqn: FourSeriesQuass}), then
	\[
	\| f -  S_Kf\|_\infty
	\leq
	\left(
	\frac{Area(\mathbb{S}^{N-1}) N^r}{K^{2r- N}(2r-N)}
	\sum_{n=1}^N \big\| \partial^r_nF  -  S_K \partial^r_nF \big\|^2_{L^2}
	\right)^{1/2}
	\] which goes to zero faster than $\frac{1}{K^{r - \frac{N}{2}}}$ as $K\to \infty$.
\end{theorem}

\subsection{Persistent Homology}
\label{section: Background persHom}
As we mentioned in  Section \ref{section:Intro},
persistent homology is a tool from Topological Data Analysis used to
study the evolution of topological features in filtered spaces.
Indeed, for any  filtration $\mathcal{K} = \{K_\epsilon\}_{\epsilon \geq 0}$, taking homology in dimension $j \in \N$ and coefficients
in a field $\F$ yields a family
\[
H_j(\mathcal{K}; \F) =
\left\{\;
T_{\epsilon,\epsilon'} :
H_j(K_\epsilon; \F) \longrightarrow H_j(K_{\epsilon'} ;\F)
\;\; , \;\;
\epsilon \leq \epsilon'
\;\right\}
\]
of $\F$-vector spaces and linear transformations $T_{\epsilon,\epsilon'}$ induced by
the inclusion maps $K_\epsilon \hookrightarrow K_{\epsilon'}$, for $\epsilon \leq \epsilon'$.
The $j$-th \emph{persistent homology groups} are
\begin{equation}\label{eqn: PersHomoGroup}
	H^{\epsilon,\epsilon'}_j(\mathcal{K}; \F)
	:=
	\mathsf{Img}
	\left(
	T_{\epsilon,\epsilon'}
	\right)
\end{equation}
and their dimension over $\F$ are the \emph{persistent Betti numbers}
\begin{equation}\label{eqn: PerBettiNum}
	\beta_j^{\epsilon, \epsilon'}(\mathcal{K})
	:=
	\mathsf{rank}(T_{\epsilon,\epsilon'})
	=
	\mathsf{dim}_\F
	\left(
	H_j^{\epsilon,\epsilon'}(\mathcal{K};\F)
	\right).
\end{equation}

If   $\beta_j^{\epsilon, \epsilon}(\mathcal{K}) < \infty$ for every $\epsilon$---i.e., if $H_j(\mathcal{K}; \F)$ is  \emph{pointwise-finite}---then a theorem
of Crawley-Boevey \cite{crawley2015decomposition} contends that the isomorphism type
of $H_j(\mathcal{K};\F)$ is uniquely determined by a multiset of
intervals $I \subset [0,\infty]$, called the \emph{barcode}
of $H_j(\mathcal{K};\F)$, and denoted $\bcd_j(\mathcal{K})$.
The (undecorated) \emph{persistence diagram} $\dgm_j(\mathcal{K})$,
on the other hand,
is the multiset of pairs $(a,b)$
resulting from taking the endpoints $a\leq b$ of
the intervals in $\bcd_j(\mathcal{K})$.
In terms of persistent Betti numbers, one can check  that
\begin{equation}\label{eqn: BettiCount}
	\beta_j^{\epsilon, \epsilon'}(\mathcal{K})
	=
	\#
	\Big\{
	I \in \bcd_j(\mathcal{K}) \; \mid  \; [\epsilon,\epsilon'] \subset I
	\Big\}
\end{equation}
where cardinality   ($\#$) on the right hand side is that of  multisets.
If all intervals in $\bcd_j(\mathcal{K})$ are of the same type
(i.e., all open, closed, right/left open), then
\begin{equation}\label{eqn: BettiCount2}
	\beta_j^{\epsilon, \epsilon'}(\mathcal{K})
	=
	\#
	\Big\{
	(a,b) \in \dgm_j(\mathcal{K}) \; \mid  \; a \prec_\ell \epsilon \leq \epsilon' \prec_r b
	\Big\}
\end{equation}
where $\prec_\ell $ and $ \prec_r$ are chosen   to
coincide with the interval type of $\bcd_j(\mathcal{K})$.

The pointwise-finite hypothesis on $H_j(\mathcal{K};\F)$ can be relaxed to   $\beta_j^{\epsilon,\epsilon'}(\mathcal{K}) < \infty$
for all $\epsilon < \epsilon'$; this is called being
$\mathsf{q}$-\emph{tame}, and is a condition satisfied by the persistent homology of the Rips filtration (defined in Eq. (\ref{eqn: RipsFilt})) of any totally bounded metric space \cite[Proposition 5.1]{chazal2014persistence}.
It is known that barcodes and persistence diagrams can be defined in the
$\mathsf{q}$-tame case in such a way that Eq. (\ref{eqn: BettiCount}) (and also Eq. (\ref{eqn: BettiCount2}) if all intervals are of the same type) is still valid
\cite[Corollary 3.8, Theorem 3.9]{chazal2016observable}.
As a result, and when $(X,\dd_X)$ is totally bounded,
we have well-defined  Rips persistence diagrams
\[
\dgm_j^\mathcal{R}(X,\dd_X) :=
\dgm_j(\mathcal{R}(X,\dd_X))
\]
for every $j\in \N$.

A bit more is true: these diagrams
are well-behaved in the sense that they are stable under Gromov-Hausdorff perturbations on $(X,\dd_X)$.
Here is what this means.
The \emph{Hausdorff} distance in  a metric space $(\mathbb{M},\dd)$  between two
bounded and non-empty subsets $X,Y\subset \mathbb{M}$
is defined as
\[
\dd_H^{\mathbb{M}}(X,Y) := \inf \left\{\delta > 0  \mid X \subset Y^{(\delta)} \mbox{ and } Y \subset X^{(\delta)}\right\}.
\]
Here $X^{(\delta)}$ (resp. $Y^{(\delta)}$) is the union
of   open balls in $\mathbb{M}$ of radius $\delta > 0$
centered at points in $X$ (resp. $Y$).
Also, when the ambient metric space is clear,
the notation $\dd_H^{\mathbb{M}}(X,Y)$ is simplified to
$\dd_H(X,Y)$.
The \emph{Gromov-Hausdorff distance}, on the other hand,
is defined for  bounded metric spaces $(X,\dd_X)$, $(Y,\dd_Y)$ as
\[
\dd_{GH}\left((X,\dd_X), (Y,\dd_Y)\right)
:= \inf_{\mathbb{M}, \varphi, \psi}
\dd^{\mathbb{M}}_H
\big(
\varphi(X)
\; , \;
\psi(Y)
\big)
\]
where the infimum runs over all metric spaces
$(\mathbb{M}, \dd)$, and all isometric embeddings
$\varphi: (X,\dd_X) \hookrightarrow (\mathbb{M}, \dd)$,
$\psi : (Y,\dd_Y) \hookrightarrow (\mathbb{M}, \dd)$.
In particular, if $X,Y \subset (\mathbb{M}, \dd)$, then
\begin{equation}\label{eqn: GHvsH}
	\dd_{GH}((X, \dd\vert_{X}), (Y, \dd\vert_{Y}))
	\leq
	\dd_H^\mathbb{M}(X,Y).
\end{equation}

The Gromov-Hausdorff distance is  a measure
of similarity between bounded metric spaces;
in fact it is a pseudometric,
which   is zero if and only if the completions of the metric
spaces involved are isometric.
The stability of Rips persistence diagrams, on the other hand,
is an inequality comparing the Gromov-Hausdorff distance
between the input metric spaces, and a notion
of distance  between their persistence diagrams called the \emph{bottleneck distance}.
This distance is defined as follows:
two persistence diagrams
$\dgm$ and $\dgm'$ are said to be $\delta$-matched, $\delta > 0$, if there is a bijection
$\mu : A \longrightarrow A'$ of multisets
$A\subset \dgm$ and $A' \subset \dgm'$
for which:
\begin{enumerate}
	\item If $(x,y) \in A$ and $(x', y')  = \mu(x,y)$,
	then $\max\left\{|x-x'|, |y - y'| \right\} < \delta$\\[-.2cm]
	\item If $(x,y) \in (\dgm\smallsetminus A) \cup (\dgm' \smallsetminus A') $ then $y -x < 2\delta$
\end{enumerate}
The bottleneck distance between $\dgm$ and $\dgm'$ is
\begin{equation}\label{eqn: DefBottleNeck}
	\dd_B
	\big(\dgm, \dgm')
	:=
	\inf \,
	\left(
	\big\{
	\delta> 0 \mid
	\mbox{ $\dgm$ and $\dgm'$ are $\delta$-matched }
	\big\}
	\cup \{\infty\}
	\right).
\end{equation}
Finally,  the stability of Rips
persistent homology \cite[Theorem 5.2]{chazal2014persistence}
contends that
\begin{equation}\label{eqn: ThmStability}
	\dd_B
	\left(
	\dgm_j^\mathcal{R}(X,\dd_X),
	\dgm_j^\mathcal{R}(Y,\dd_Y)
	\right)
	\;\leq\;
	2\dd_{GH}
	\big(
	(X,\dd_X),
	(Y,\dd_Y)
	\big)
\end{equation}
for $(X,\dd_X)$ and $(Y,\dd_Y)$ totally bounded.

\section{Fourier Approximations of Sliding Window Persistence}\label{section:fourierseries}
With the preliminaries out of the way, we now move onto studying the Rips
persistent homology of sliding window point clouds from quasiperiodic functions.
Thus far we have that if $f$ is quasiperiodic and its parent function $F$ has enough regularity,
then $f$ can be uniformly approximated by the  truncated series $S_K f$.
This is the content of Theorem \ref{thm: UnifConvFourierQuasi},
and in  particular says that the higher the  smoothness of $F$, then the faster the degree of approximation $S_K f \rightarrow f$.
We will see next  that these results can be readily bootstrapped to the level
of sliding window point clouds, and hence to statements about Rips persistence diagrams.

\begin{theorem}\label{approx}
	Let $f$ be quasiperiodic with parent function	$F \in C^r(\T^N)$, $ r > \frac{N}{2}$.
	If
	\[\SW_{d,\tau} f = SW_{d,\tau} f(T)
	\mbox{\;\;\; and \;\;\;}
	\SW_{d,\tau} S_K f = SW_{d,\tau} S_K f(T)
	\;\;\; ,\;\;\;
	T\subset \R
	\] are the   sliding window point clouds
	of $f$ and $S_K f$, respectively,
	then the Hausdorff distance between them satisfies
	\begin{eqnarray*}
		\dd_H
		\big( \SW_{d,\tau}f \, , \, \SW_{d,\tau}S_Kf \big)
		&\leq& \sqrt{d+1}\|f - S_K\|_\infty
		\\
		&\leq&
		\left(
		\frac{Area(\mathbb{S}^{N-1})(d+1) N^r}{K^{2r- N}(2r-N)}
		\sum_{n=1}^N \big\| \partial^r_nF  -  S_K \partial^r_nF \big\|^2_{L^2}
		\right)^{1/2}
	\end{eqnarray*}
	which goes to zero faster than $\frac{1}{K^{r - \frac{N}{2}}}$ as $K\to \infty$.
\end{theorem}
\begin{proof}
	Let $\epsilon > \sqrt{d+1}\|f - S_K f\|_\infty $, and let $t\in T$.
	Then
	\[
	\|SW_{d,\tau}f(t) - SW_{d,\tau}S_K f(t) \|_2 \;\;\leq \;\sqrt{d+1}\|f  - S_K f\|_\infty  \;<\; \epsilon
	\]
	which implies that $\epsilon$ satisfies both
	\begin{equation}\label{eqn: HaussBound}
		\SW_{d,\tau} f \subset \big(\SW_{d,\tau} S_K f\big)^{(\epsilon)}
		\;\;\;\;
		\mbox{and}
		\;\;\;\;
		\SW_{d,\tau}S_K f \subset \big(\SW_{d,\tau} f\big)^{(\epsilon)}.
	\end{equation}
	Since the Hausdorff distance in $\C^{d+1}$ between $\SW_{d,\tau} f$ and $\SW_{d,\tau} S_K f$ is the infimum over all  $\delta > 0$ satisfying Eq. (\ref{eqn: HaussBound}), then we have that
	\[
	\dd_H \big(\SW_{d,\tau} f \, , \, \SW_{d,\tau} S_K f\big)
	\;\leq\;
	\epsilon.
	\]
	Because this is true for any $\epsilon > \sqrt{d+1}\|f - S_K f\|_\infty$,
	then
	\[
	\dd_H \big(\SW_{d,\tau} f \, , \, \SW_{d,\tau} S_K f\big) \leq \sqrt{d+1}\|f - S_K f\|_\infty
	\]
	and  the bound from
	Theorem \ref{thm: UnifConvFourierQuasi} finishes the proof.
\end{proof}

Using the stability of Rips persistent homology (Eq. (\ref{eqn: ThmStability})),
we can readily  bound  the bottleneck distance
between the corresponding persistence diagrams:

\begin{corollary}\label{Corollary: DGMapproximation}
	With the  same hypotheses of Theorem \ref{approx}, and for all $j\in \N$,
	\begin{eqnarray*}
		\dd_B \left( \dgm^\mathcal{R}_j(\SW_{d,\tau}f) , \dgm^\mathcal{R}_j(\SW_{d,\tau}S_Kf) \right)
		&& \\
		&&
		\hspace{-3.2cm}
		\leq \hspace{.2cm}
		2\sqrt{d+1}\|f - S_K f\|_\infty \\
		&&
		\hspace{-3.2cm}
		\leq \hspace{.2cm}
		2\left(
		\frac{Area(\mathbb{S}^{N-1})(d+1) N^r}{K^{2r- N}(2r-N)}
		\sum_{n=1}^N \big\| \partial^r_nF  -  S_K \partial^r_nF \big\|^2_{L^2}
		\right)^{1/2}
	\end{eqnarray*}
	and thus  goes to zero faster than $\frac{1}{K^{r - \frac{N}{2}}}$ as $K\to \infty$.
\end{corollary}

The main point of these approximation results is that
studying $\dgm_j^\mathcal{R}(\SW_{d,\tau} f)$
can be reduced to understanding $\dgm_j^\mathcal{R}(\SW_{d,\tau} S_K f)$
and its asymptotes as $K\to \infty$.
This
is a vastly more accessible simplification, as we will see shortly.

\section{The geometric structure of $\SW_{d,\tau}S_K f$}\label{section: structuretheorem}	
Our next goal is to show that for suitable  choices of
$d,K \in \N$, $T\subset \R$,  and $\tau>0$,     the closure of the sliding window
point cloud $\SW_{d, \tau}S_Kf = SW_{d,\tau} S_Kf(T)$ in $\C^{d+1}$,
is homeomorphic to an $N$-torus.
Indeed, for $F\in C^r(\T^N)$ and $K\in \N$, let
\[
\mathsf{supp}\left(\widehat{F}_K\right) =
\left\{\kk \in I_K^N \mid  \widehat{F}(\kk) \neq 0\right\}
\]
denote the support of the Fourier transform $\widehat{F}$ restricted to the square box $I^N_K$.

\begin{lemma}\label{lemma: lattice}
	Let $f(t) = F(\omega t)$ be quasiperiodic with frequency vector $\omega \in \R^N$,
	and parent function $F\in C^r(\T^N)$, $r > \frac{N}{2}$.
	Then, for all large enough $K\in \N$,   $\mathsf{supp}\left(\widehat{F}_K\right)$
	spans an $N$-dimensional $\Q$-vector space.
	
\end{lemma}
\begin{proof}
	The first thing to note is that since
	\[
	\mathsf{supp}\left(\widehat{F}\right)
	=
	\bigcup_{K\in \N}
	\mathsf{supp}\left(\widehat{F}_K\right)
	\subset
	\Z^N,
	\]
	then $V = \mathsf{span}_\R\left(\mathsf{supp}\left(\widehat{F}\right)\right)$
	is an $\R$-linear subspace of $\R^N$.
	It follows that
	\[
	L = \mathsf{span}_\Z\left(\mathsf{supp}\left(\widehat{F}\right)\right)
	\]
	is an additive discrete subgroup of $V$, and therefore a lattice \cite[Theorem 6.1]{stewart2015algebraic}
	of dimension $n \leq \mathsf{dim}_\R(V) \leq N$. It suffices to show that $n = N$.
	
	Let $\zz_1,\ldots, \zz_n \in L$ be so that $L = \mathsf{span}_{\Z}\{\zz_1,\ldots, \zz_n\}$.
	Incommensurability of $\omega$ implies that $\tilde{\omega}_j  = \<\zz_j , \omega\> $, $j=1,\ldots, n$, are $\Q$-linearly independent,
	and we can assume without loss of generality that $\tilde{\omega}_j >0$; otherwise replace $\zz_j$ by $-\zz_j$
	as a basis element for  $L$.
	Hence, $\tilde{\omega} = \left(\tilde{\omega}_1, \ldots, \tilde{\omega}_n\right) \in \R^n$
	is a vector of incommensurate frequencies.
	
	For   $\tt \in \T^n$ let
	\[
	G(\tt) :=
	\sum_{\aa \in \Z^n} \widehat{F}(a_1\zz_1 +\cdots + a_n \zz_n) e^{i\<\aa, \tt\>}
	\]
	which converges uniformly in $\tt\in \T^n$ since the Fourier coefficients $\widehat{F}(\kk)$
	are  absolutely summable (Theorem \ref{thm: coefficientbounds}).
	Therefore $G \in C(\T^n)$, and
	thus
	\begin{eqnarray*}
		f(t) &=& \sum_{\kk \in \Z^N} \widehat{F}(\kk) e^{i\<\kk, \omega\>t} \\
		&=& \sum_{\aa \in \Z^n} \widehat{F}(a_1\zz_1 +\cdots + a_n \zz_n) e^{i\<\aa, \tilde{\omega}\>t} \\
		&=&G\left( \tilde{\omega}t\right)
	\end{eqnarray*}
	which shows that $G$ is also a parent function for $f$, with $\tilde{\omega}$ as the corresponding
	frequency vector. Since the dimension $N$ of the frequency vector for $f$ is assumed
	to be minimal (Remark \ref{rmk: MinmalFreqDim}), then $n= N$,
	completing the proof.
\end{proof}

Now, if we write
$\mathsf{supp}\left(\widehat{F}_K\right)= \left\{\kk_1,  \ldots, \kk_{\alpha}\right\}$, for
$ 1\leq \alpha\leq (1 + 2K)^N$, then
\begin{align}\label{eqn:SWrelationshipX}
	SW_{d,\tau}S_Kf(t) = \Omega_{K,f} \cdot x_{K,f}(t)
\end{align}
where
\begin{align}\label{eqn: little_XZf}
	x_{K,f}(t)=
	\begin{bmatrix}
		\widehat{F}(\kk_1) e^{i \< \kk_1, \omega \> t} \\
		\vdots\\
		\widehat{F}(\kk_{\alpha}) e^{i \< \kk_{\alpha}, \omega \> t}
	\end{bmatrix}
	\in \C^\alpha
\end{align}
and $\Omega_{K,f}$ is the Vandermonde  $(d+1) \times \alpha $ matrix
\begin{align}\label{eqn: Omega}
	\Omega_{K,f}
	=
	\begin{bmatrix}
		1 & \cdots & 1 \\
		e^{i\langle\mathbf{k}_1, \omega\rangle \tau } & \cdots & e^{i\langle\mathbf{k}_{\alpha}, \omega\rangle \tau } \\
		\vdots &  & \vdots \\
		e^{i\langle\mathbf{k}_1, \omega\rangle \tau d} & \cdots & e^{i\langle\mathbf{k}_{\alpha}, \omega\rangle \tau d}
	\end{bmatrix}
\end{align}
with nodes $e^{i\langle\mathbf{k}_1, \omega\rangle \tau } , \ldots,  e^{i\langle\mathbf{k}_{\alpha}, \omega\rangle \tau } \in S^1
\subset \C$ \cite{aubel2019vandermonde}.
We define $\X_{K,f}$ to be the collection of vectors $x_{K,f}(t)$ as above in Eq. (\ref{eqn: little_XZf}):
\begin{equation}\label{eqn: big_XKf}
	\X_{K,f} =
	\left\{
	x_{K,f}(t) \mid t\in \R
	\right\}
	\subset\C^\alpha.
\end{equation}
The decomposition in Eq. (\ref{eqn:SWrelationshipX}) with Lemma \ref{lemma: lattice} yields conditions on the parameters $K, d,\tau$
under which the sliding window point cloud $\SW_{d,\tau} S_K f$ is dense in
a torus of the appropriate dimension. Indeed,

\begin{theorem}\label{thm: Structure}
	Let $f(t) = F(t\omega) $ be quasiperiodic with frequency vector $\omega \in \R^N$
	and parent function $F\in C^r(\T^N)$, $r > \frac{N}{2}$.
	Let
	\[
	\mathsf{supp}(\widehat{F}_K ) = \{\kk_1,\ldots, \kk_\alpha\}
	\]
	and assume that $\tau > 0$ is not an integer multiple of
	$\frac{2\pi}{\<\kk_n - \kk_m , \omega\>}$ for  $1 \leq n < m \leq \alpha$.
	If $K \in \N$ is large enough so that
	$\mathsf{supp}(\widehat{F}_K)$ spans an $N$-dimensional $\Q$-vector space, and $d \geq \alpha -1 $, then the sliding window point cloud
	\[
	\SW_{d,\tau} S_K f = SW_{d,\tau} S_K f(\R)
	\]
	is dense in an $N$-torus embedded in $\C^{d+1}$.
\end{theorem}
\begin{proof}
	The first thing to note is that since $\tau > 0 $ is not an integer multiple of any $\frac{2\pi}{\<\kk_n - \kk_m , \omega\>}$,
	$1 \leq n < m \leq \alpha$,
	then the points $e^{i\langle\mathbf{k}_1, \omega\rangle \tau } , \ldots,  e^{i\langle\mathbf{k}_{\alpha}, \omega\rangle \tau } \in S^1$
	are all  distinct.
	Thus, the Vandermonde matrix $\Omega_{K,f}$ is full rank.
	This can be checked via induction on $\alpha$, by showing that the determinant of an
	$\alpha\times \alpha$ Vandermonde matrix with   nodes $\zeta_1,\ldots, \zeta_\alpha$
	is $\prod\limits_{1\leq n < m \leq \alpha} (\zeta_m - \zeta_n)$.
	Combining this observation with  $d+1\geq \alpha$,
	implies that $\Omega_{K,f} : \C^{\alpha} \longrightarrow \C^{d+1}$
	is injective as a linear transformation.
	
	Now,  Corollary \ref{Corollary: Kronecker} with $\beta_1 = \<\kk_1,\omega\> ,\ldots, \beta_\alpha = \<\kk_\alpha, \omega\>$,
	together with  Lemma
	\ref{lemma: lattice}, imply that for all large enough $K\in\N$ the point cloud $\X_{K,f}$ (defined in Eq. (\ref{eqn: big_XKf}))
	is dense in an $N$-torus embedded in $\C^{\alpha}$.
	The result follows from $\Omega_{K,f}$ being a linear homeomorphism onto its image.
\end{proof}

\begin{corollary} With the same hypotheses of Theorem \ref{thm: Structure}, and if
	$\{\pi,\omega_1,\ldots, \omega_N\}$ is incommensurate, then the sliding window point cloud
	\[
	\SW_{d,\tau} S_K f = SW_{d,\tau} S_Kf( \Z)
	\]
	is dense in an  $N$-torus embedded in $\C^{d+1}$.
\end{corollary}
\begin{proof} If $\kk_1,\ldots, \kk_N \in \Z^N$ are $\Q$-linearly independent,
	then incommensurability of $\{\pi,\omega_1,\ldots, \omega_N\}$ implies incommensurability of
	$\{\pi, \<\kk_1,\omega\>, \ldots, \<\kk_N,\omega\>\}$. The result
	follows in the same way as Theorem \ref{thm: Structure}, but using the integer version of Kronecker's theorem
	as starting point (see Remark \ref{rmk: IntegerKronecker}).
\end{proof}

We would like to emphasize that the condition on $\tau$ in Theorem \ref{thm: Structure}
only guarantees the topology of an $N$-torus.
Preserving the geometric structure as much as possible when going from $\X_{K,f}$ to $\SW_{d,\tau} S_K f$,
and consequently amplifying the toroidal features in
$\dgm_j^\mathcal{R}\left(\SW_{d,\tau} S_K f\right)$, requires specific optimizations on $\tau$.

\section{Parameter selection: how to optimize  $d$ and $\tau$?}\label{section: parameterselection}

\subsection{The embedding dimension}\label{subSection:d}
In practice,
the diagrams
\[
\dgm_j^\mathcal{R}\left(SW_{d,\tau} f(T)\right) \;\;\;\; ,\;\;\;\;  T\subset \R  \;\;\mbox{finite}
\]
are computed
as  approximations to those of  $SW_{d,\tau} f(\R)$;
the latter set is relatively compact,
and hence the stability theorem (see Eq. (\ref{eqn: ThmStability})) implies that finite samples
provide arbitrarily good approximations.
The difficulty lies in that as $d \to \infty$, it becomes necessary
for $T$ to also grow in order to overcome
the curse of (ambient space) dimensionality, and provide appropriate geometric recovery \cite{radovanovic2010hubs}.
This is problematic since the Rips filtration
grows exponentially in the number of points,
and the matrix reduction algorithm for computing persistent homology is in the worst case cubic in the number of simplices \cite{morozov2005persistence}.
It is thus desirable for $d$ to be as small as possible.
With this   and Theorem \ref{thm: Structure} in mind, we propose the following procedure for choosing $d$:
Let $K$ be the smallest integer so that
$\mathsf{supp}(\widehat{F}_K)$ spans an $N$-dimensional
vector space over $\Q$, and let
$d$ be the cardinality ($\alpha$) of $\mathsf{supp}(\widehat{F}_K)$.
When $f$ is given numerically as a
potentially noisy time
series sampled at finitely many evenly spaced
time points, then  $d$ can be estimated  as the number
of prominent peaks in the spectrum of
$f$.

\begin{remark}\label{rmk: anotherd}
	The structure theorems for both periodic functions \cite[Theorem 5.6]{sw1pers} and quasiperiodic functions (Theorem \ref{thm: Structure}) only require $d\geq \alpha - 1$. While the choice   $d = \alpha-1$ works for periodic signals in practice, we will demonstrate in Example \ref{example: tau} that   $d = \alpha$ is
	preferable in the quasiperiodic case.
	This discrepancy arises in the computation of the time delay $\tau$.
	Indeed, while for periodic functions  there is a clear closed-form choice of $\tau$, it turns out that this is typically not possible in the quasiperiodic case.
	We will investigate how in what follows.
\end{remark}

\subsection{The time delay}\label{subSection:tau}
One way in which $\tau$ controls the shape
of $\SW_{d,\tau} S_K f$ is via the condition number (i.e., the largest singular value divided by the smallest singular value) of the
Vandermonde matrix $\Omega_{K,f}$ (defined in Eq. (\ref{eqn: Omega})).
Indeed, when this number is much larger than 1
and the singular subspaces from the smallest singular
values of $\Omega_{K,f}$ intersect $\overline{\X_{K,f}}$ transversally, then the persistence of the toroidal features
of $\SW_{d,\tau} S_K f$ localized along these directions
can be greatly diminished.
One can avoid this problem by selecting a $\tau >0$ promoting orthogonality
between the columns $\uu_1,\ldots, \uu_\alpha$ of $\Omega_{K,f}$.
Indeed,
mutual orthogonality together with $\|\uu_1\| = \cdots = \|\uu_\alpha\| = \sqrt{d+1}$ would imply
that $\Omega_{K,f}$ is $\sqrt{d+1}$ times a linear isometry.
Such a transformation would have condition number equal to 1, and would preserve the persistent features of $\X_{K,f}$ (these are described in Theorem \ref{theorem: longbarsX}).
That said, exact mutual orthogonally of the $\uu_j$'s is not possible in
general, for if $N \geq 3$, then
$\<\uu_1,\uu_2\> = \<\uu_1,\uu_3\> = 0$ implies that
there exist $m,m'\in \Z$ satisfying
\begin{eqnarray*}
	\<\kk_1 - \kk_2 ,\omega\>\tau(d+1) &=& 2\pi m \\
	\<\kk_1 - \kk_3, \omega\>\tau(d+1) &=& 2\pi m'
\end{eqnarray*}
which in turn would imply
\[
m'\<\kk_1 - \kk_2, \omega\> = m\<\kk_1 - \kk_3, \omega\>
\]
contradicting either the linear independence of
$\kk_1,\kk_2 ,\kk_3$, or the incommensurability of $\omega$.
We will settle for the next best option:
to let $\tau$ be so that
the $\uu_j$'s are, in average, as orthogonal as possible.
In other words, we will choose $\tau$ as
a minimizer over
$[0, \tau_\mathsf{max}]$ of the scalar function
\begin{equation}\label{eqn: Gopt}
	\Gamma(x) :=
	\sum_{1 \leq j < \ell \leq \alpha }
	\left\vert
	1 + e^{i\<\kk_j - \kk_\ell, \omega\>x}
	+ \cdots +
	e^{i\<\kk_j - \kk_\ell, \omega\>xd}
	\right\vert^2
\end{equation}
which is exactly the sum of squared magnitudes of the inner products $\<\uu_j, \uu_\ell\>$ between the columns
of the Vandermonde matrix $\Omega_{K,f}$.
The thing to note is that when $f$ is given
as a   noisy finite sample,
then the inner products $\<\kk_j,\omega\>$
can be estimated as the frequency locations of the prominent peaks
in the spectrum of $f$.

\begin{example}\label{example: tau}
	As an illustration of our parameter selection procedure, let
	\[
	f(t) = e^{i t} + e^{i \sqrt{2}t} + e^{i \sqrt{3}t}  \;\;\; ,\;\;
	0 \leq t \leq 400.
	\]
	The real and imaginary part of this function are shown in Figure \ref{fig:expRealImgPart}
	below.
	\begin{figure}[!htb]
		\centering
		\subfloat{
			\includegraphics[width=\textwidth]{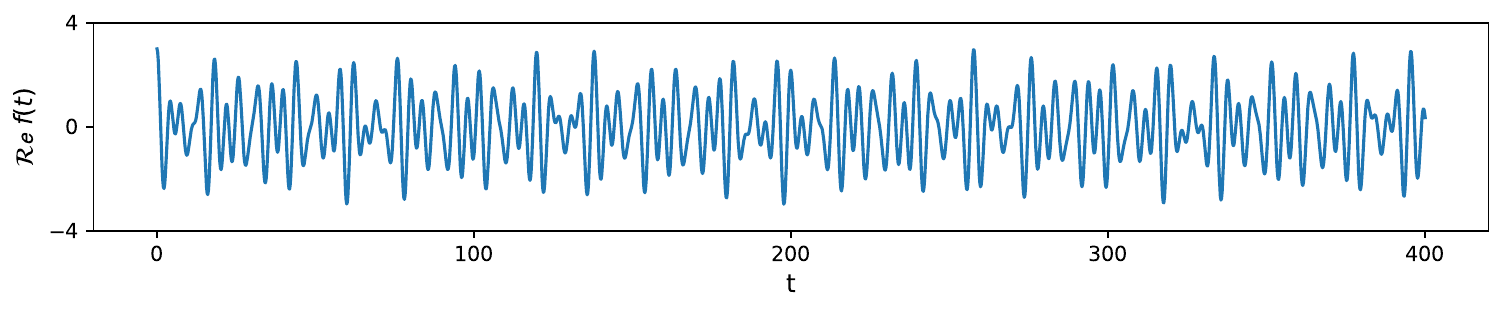}
		}
		
		\subfloat{
			\includegraphics[width=\textwidth]{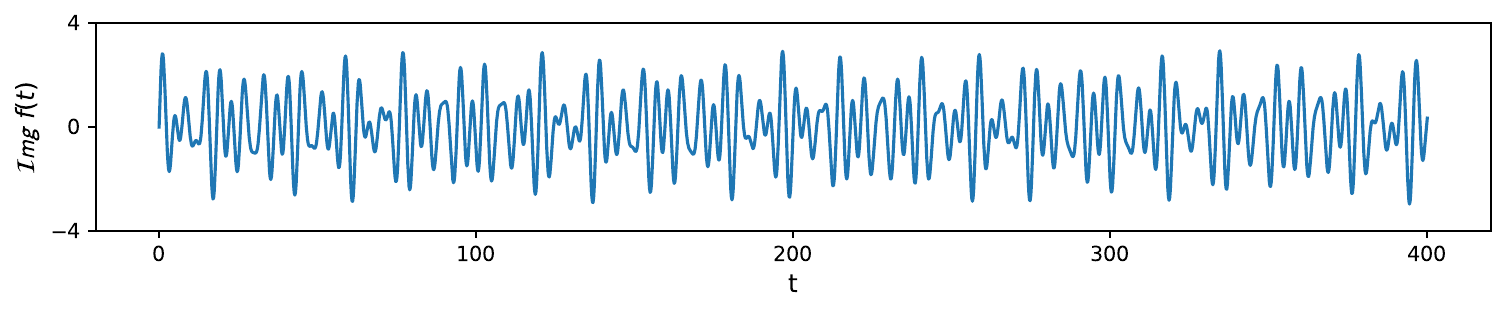}
		}
		\caption{Real (top) and imaginary (bottom) part of
			the function $f(t) = e^{it} + e^{i\sqrt{2}t} + e^{i\sqrt{3}t}$, $0\leq t \leq 400$.}
		\label{fig:expRealImgPart}
	\end{figure}
	
	\noindent It can be readily checked that $\omega = \left(1, \sqrt{2}, \sqrt{3}\right)$ is the frequency vector
	for $f$, and
	\[
	\mathsf{supp}(\widehat{F}_K) = \big\{(1,0,0), (0,1,0), (0,0,1)\big\}
	\;\;\; ,\;\; K\geq 1.
	\]
	Following the discussion from Section \ref{subSection:d} we
	let $d = 3$ (the cardinality of $\mathsf{supp}(\widehat{F}_K)$) and compute
	$\dgm_j^\mathcal{R}\left(SW_{d,\tau} f(T)\right)$ for $T\subset [0,400]$
	in dimensions $j=1,2$ as follows.
	We begin by evaluating $SW_{d,\tau}f(t)$ at 2,000 evenly spaced points in $[0,400]$,
	and then further subsample this point cloud by selecting 800 points via \verb"maxmin" sampling.
	That is, we pick   $t_1\in\widetilde{T} = \left\{ \frac{n}{5} \mid n=0,\ldots, 2000\right\}$ uniformly at random, and if $t_1 ,\ldots, t_{\ell} \in \widetilde{T}$ have been selected, then we let
	\[
	t_{\ell +1 } = \argmax_{t \in \widetilde{T}} \min
	\Big\{
	\|SW_{d,\tau} f(t) - SW_{d,\tau}f(t_1)\|
	, \ldots,
	\|SW_{d,\tau} f(t) - SW_{d,\tau}f(t_{\ell})\|
	\Big\}.
	\]
	This inductive process continues until the sampling set
	$T = \{t_1, \ldots, t_{800}\} \subset \widetilde{T}$
	is constructed, and then we compute the Rips persistence
	diagrams of $SW_{d,\tau} f(T)$ in dimensions $j=1,2$, coefficients in
	$\Z_2$, and two choices of time delay: $\tau = 16.458$ and $\tau = 49.325$.
	The resulting persistence diagrams are shown in Figure \ref{fig:dgmsDiffTaus} below.
	\begin{figure}[!htb]
		\centering
		\includegraphics[width=0.98\textwidth]{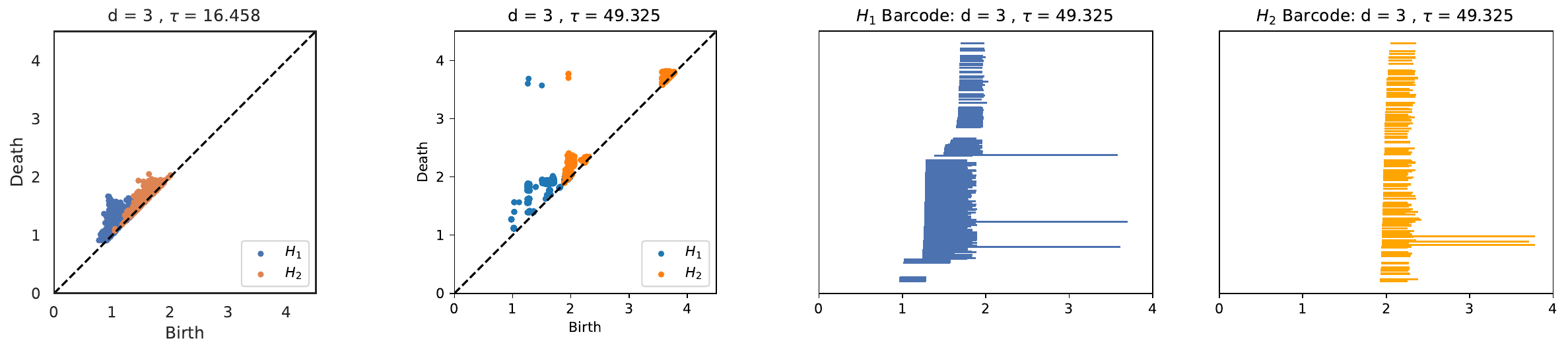}
		\caption{Rips persistence diagrams of $SW_{d,\tau} f(T)$ in dimensions $j=1$ (blue),  $j=2$ (orange), $\F = \Z_2$,  $d =3$, $\tau = 16.458$ (first) and $\tau = 49.325$ (second). Computations performed with \texttt{Ripser.py} \cite{ctralie2018ripser}. We also provide the barcode representations in dim $1$ (third) and dim $2$ (fourth) to substantiate that there are indeed three strong classes in both dimensions.}\label{fig:dgmsDiffTaus}
	\end{figure}
	
	We note that the  \verb"maxmin" sampling is used here because it selects subsample points in way that prevents clustering. This can be observed in the equation above: the time $t_{\ell +1 }$ selected corresponds to the point $SW_{d,\tau} f(t_\ell)$ which is the farthest from the already chosen set $\{SW_{d,\tau}f(t_1), SW_{d,\tau}f(t_2), \dots,SW_{d,\tau}f(t_\ell)\}$. 
	
	For this particular example we expect persistence diagrams   consistent
	with a   3-torus---i.e., 3 strong classes in dimension 1, and 3 strong classes in dimension 2--- since there are three linearly independent
	frequencies: $1,\sqrt{2}$ and $ \sqrt{3}$.
	That said, and as  Figure \ref{fig:dgmsDiffTaus} shows, a poor choice of time delay (e.g., $\tau = 16.458$) can completely obscure these toroidal features with sampling artifacts (points near the diagonal).
	This 
	stresses  how important the need for delay optimization can be.
	
	A broader picture of how the persistence of the top 3 features in each dimension
	varies with $\tau$ is shown in Figure \ref{fig:PersVsTau}.
	\begin{figure}[!htb]
		\centering
		\subfloat{
			\includegraphics[width=\textwidth]{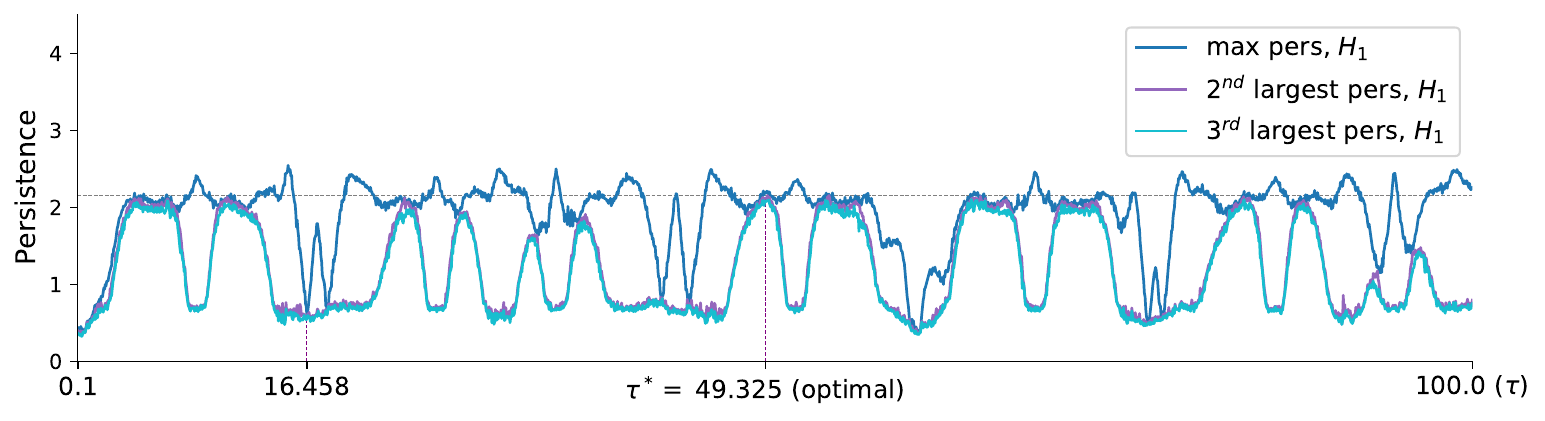}
		}
		
		\subfloat{
			\includegraphics[width=\textwidth]{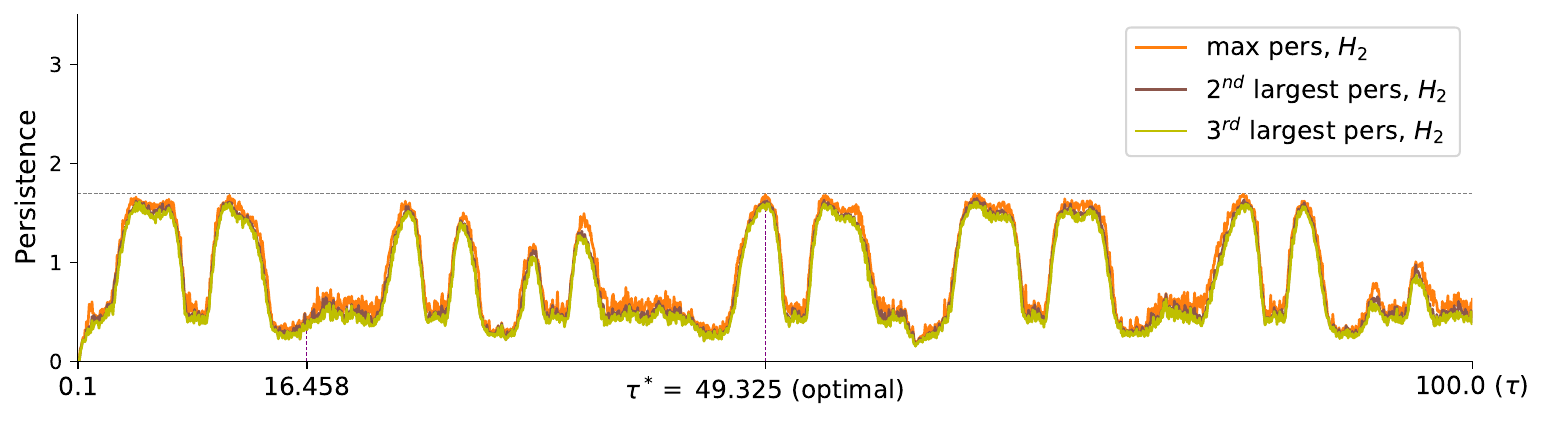}
		}
		\caption{Persistence of top 3 features in dimension 1 (top) and dimension 2 (bottom) as a function of $\tau$, for $SW_{d,\tau} f(T)$ and $d= 3$.}
		\label{fig:PersVsTau}
	\end{figure}
	
	The value $\tau = 49.325$ is optimal in the sense that it jointly maximizes the
	persistence of the top 3 features in both dimensions.
	More importantly, it is also optimal in that it is a global minimizer over $[0,100]$ for
	the function $\Gamma(x)$ (defined in Eq. (\ref{eqn: Gopt})) as described
	in Section \ref{subSection:tau}.
	We reiterate that the values $\<\kk,\omega\>$, $\kk\in \mathsf{supp}(\widehat{F}_K)$,
	needed to compute  $\tau $ as the minimizer of $\Gamma(x)$ can be estimated
	numerically as the frequency locations of the $d$ most prominent peaks in the spectrum of $f$.
	Indeed, Figure \ref{fig:fHat} shows the result of computing the Discrete
	Fourier Transform $\widehat{f}(\xi)$ of $f$ sampled at $t\in \widetilde{T}$,
	as well as the locations of the most prominent peaks in amplitude.
	\begin{figure}[!htb]
		\centering
		\includegraphics[width=\textwidth]{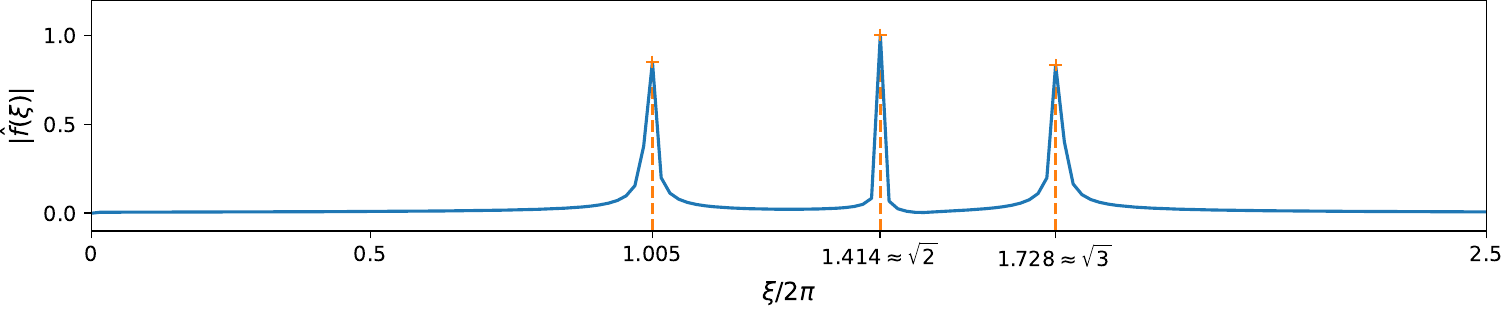}
		\caption{Modulus of the Discrete Fourier Transform for $f$ sampled at $t\in \widetilde{T}$. The locations of the most prominent peaks
			approximate the inner products  $\<\kk, \omega\>$.}\label{fig:fHat}
	\end{figure}
	
	An important thing to note is that the Discrete Fourier Transform  (DFT) by itself is known to provide only very
	rough approximations to the frequency locations of quasiperiodic functions.
	This can have   deleterious effects in the appropriate estimation of $\tau$ via minimization of $\Gamma(x)$.
	One solution is to use methods like \cite{gomez2010collocation, laskar1993frequency}, which
	leverage the DFT to produce high-accuracy frequency estimates.

	Finally, to illustrate the difference between the choices $d = \alpha$
	and $d = \alpha -1$ outlined in Remark \ref{rmk: anotherd},
	we repeat the same process above with  $d=2$.
	The persistence of the top 3 features in dimensions 1 and 2, as a function of $\tau$, is shown in Figure \ref{fig:PersVsTau_small} below.
	\begin{figure}[!htb]
		\centering
		\subfloat{ \includegraphics[width=\textwidth]{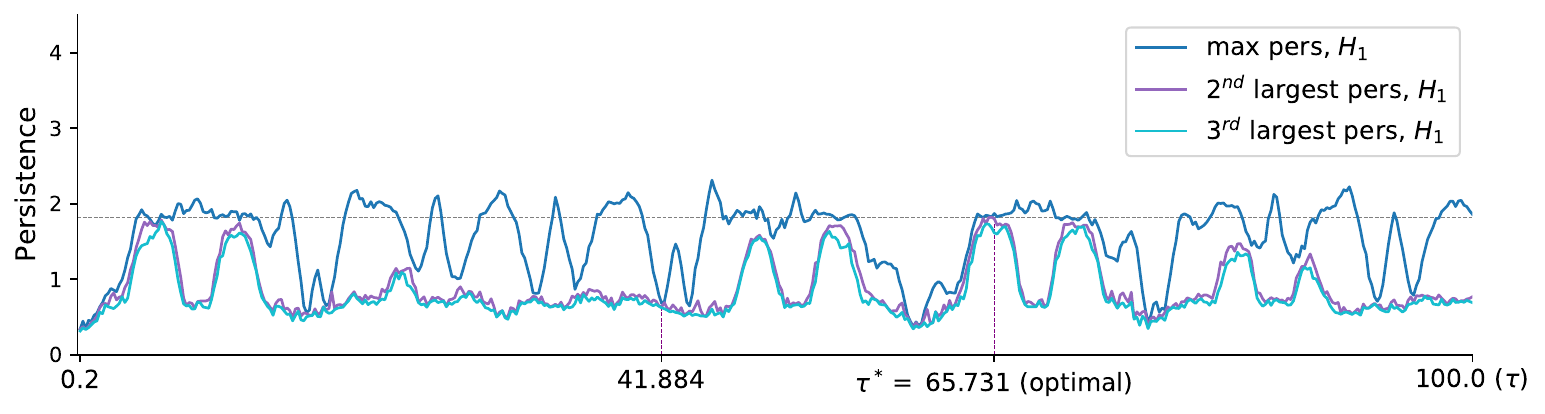}
		}
		
		\subfloat{
			\includegraphics[width=\textwidth]{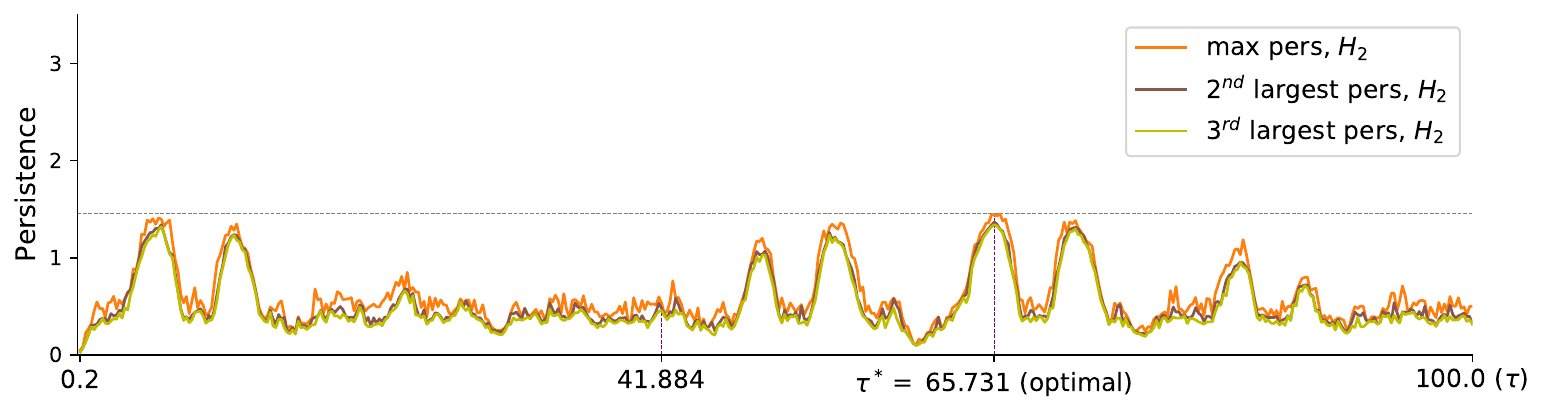}
		}
		\caption{Persistence of top 3 features in dimension 1 (top) and dimension 2 (bottom) as a function of   $\tau$, for $SW_{d,\tau} f(T)$ and $d= 2$.}
		\label{fig:PersVsTau_small}
	\end{figure}
	
	As Figure \ref{fig:PersVsTau_small} shows, the global minimizer $\tau = 65.731$ 
	of $\Gamma(x)$, $0 \leq x \leq 100$, jointly maximizes the top 3 persistence values in both dimensions.
	In particular, the underlying 3-torus topology is clearly captured by this choice of time delay.
	One thing to note  when comparing  Figure \ref{fig:PersVsTau} ($d = \alpha =3$) 
	and Figure \ref{fig:PersVsTau_small} ($d = \alpha - 1 = 2$)  
	is the number and nature of local maxima in persistence (specially in dimension 2) as a function of $\tau$.
	Indeed, $d = 3$ yields a larger number of stable local maxima; by stable 
	we mean that the values of persistence remain large in a neighborhood of a local 
	maximizer. 
	This suggests that while  $d = 2$ still captures the right underlying topology, 
	as Theorem \ref{thm: Structure} guarantees, 
	the embedding in $\C^3$ of the sliding window point cloud is nonlinear enough 
	that strong features in persistence (with the ambient Euclidean distance) occur for only very specific time delays.

\end{example}

\section{The Rips Persistent Homology of $\SW_{d,\tau}S_Kf$ and $\SW_{d,\tau}f$} \label{section: persistent homology}
We   now turn our attention to  the  persistent homology of the sliding window point clouds $\SW_{d,\tau}S_Kf$ and $\SW_{d,\tau}f$,
as well as their dependence on
both   the Fourier coefficients $\widehat{F}(\kk)$, and the parameters $K,d,\tau$.
Our aim is to establish bounds on the cardinality and persistence of
strong toroidal features in $\dgm_j^\mathcal{R}(\SW_{d,\tau} f)$.
To that end, let $K\in \N$ be so that
\[
\mathsf{supp}\left(\widehat{F}_K\right) = \{\kk_1, \ldots, \kk_\alpha\}
\;\;\; \;\;\;,\;\;\;\;\;\;
1 \leq \alpha \leq (1 + 2K)^N
\]
spans a $\Q$-vector space of dimension $N$ (Lemma \ref{lemma: lattice}).
We will further assume, after re-indexing if necessary, that  $\kk_1,\ldots, \kk_N$ are $\Q$-linearly independent  and that
\[
\vert \widehat{F}(\kk_1) \vert \geq \vert \widehat{F}(\kk_2) \vert \geq \cdots \geq \vert \widehat{F}(\kk_N)\vert > 0.
\]
With this convention, let
\begin{equation}\label{eqn: FourierTorus}
	\T^N_{\widehat{F}}
	:=
	\left\{\,
	\zz \in \C^N \; : \; \vert z_1\vert = \vert\widehat{F}(\kk_1)\vert,\ldots, \vert z_N\vert = \vert\widehat{F}(\kk_N)\vert \,
	\right\}
\end{equation}
regarded as a metric space  with the $\|\cdot\|_\infty$ distance:
\[
\|\zz - \zz'\|_\infty = \max \{\vert z_1 - z_1'\vert ,\ldots, \vert z_N - z_N' \vert\}.
\]

In order to understand the Rips persistent homology of
$\SW_{d,\tau} f$, the first step  is to clarify that of $\T^N_{\widehat{F}}$.
This involves two theorems that we  now describe.
The first is a result by Adams and Adamaszek \cite[Theorem 7.4]{vrcircle}
computing
the homotopy type
of $ R_\epsilon(S^1) $ at each scale $\epsilon > 0$.
\begin{prop}[\cite{vrcircle}]\label{prop: vrcircle}
	The Rips complex ${R}_{\epsilon}(S^1_r)$ of a circle $S^1_r \subset \C$ of radius $r$ (equipped with the Euclidean metric) is homotopy equivalent to $S^{2\ell+1}$ if and only if
	\begin{align*} r_\ell =
		2r \sin\left(\pi \frac{\ell}{2\ell+1}\right) \;< \;\epsilon \; \leq \; 2r\sin\left(\pi \frac{\ell+1}{2\ell+3}\right)  = r_{\ell+1}\;\;\; , \;\;\; \ell \in \N.
	\end{align*}
	Moreover, for all $\ell \in \N$ and $r_\ell < \epsilon \leq \epsilon' \leq r_{\ell +1}$, the inclusion $R_\epsilon(S^1_r) \hookrightarrow R_{\epsilon'}(S^1_r)$
	is a homotopy equivalence.
\end{prop}
As a consequence, the  Rips persistent homology of $(S^1_r, \vert\cdot\vert)$ is pointwise-finite---hence $\mathsf{q}$-tame---the resulting barcodes (and hence the persistence diagrams) are  singleton multisets in odd dimensions, and empty in positive even dimensions:
\begin{align}\label{eqn: RipsPersCircle}
	\bcd^\mathcal{R}_{j}\left(S^1_r, \vert\cdot\vert\right)
	=
	\begin{cases}
		\left\{
		\left(
		2r \sin\left(\pi \frac{\ell}{2\ell+1}\right) , 2r\sin\left(\pi \frac{\ell+1}{2\ell+3}\right)
		\right]
		\right\}
		&\textrm{ if } j = 2\ell+1\\[.2cm]
		\left\{
		\left(0,\infty\right)
		\right\}
		&\textrm{ if } j = 0\\[.2cm]
		\emptyset
		&\textrm{ if } j = 2\ell+2
	\end{cases}
\end{align}
The second result needed to describe $\dgm_j^\mathcal{R}(\T^N_{\widehat{F}}, \|\cdot\|_\infty)$ is a   K\"{u}nneth formula for Rips persistent homology and the maximum metric \cite[Corollary 4.6]{Kunneth}.

\begin{prop}[\cite{Kunneth}]\label{prop :RipsKunneth}
	Let $(X_1,\dd_1), \dots ,(X_N,\dd_N)$ be  metric spaces  with pointwise-finite Rips persistent homology.
	Then, for all $j\in \N$
	\[
	\bcd_j^\mathcal{R}(X_1\times \cdots \times X_N,\dd_\mathsf{max})
	=
	\left\{
	\bigcap_{n=1}^N I_{j_n}  \; \Big\vert \;
	I_{j_n} \in \bcd^\mathcal{R}_{j_n}(X_n,\dd_n) \;\;,\;
	\sum\limits_{n=1}^N j_n = j\;
	\right\}
	\]
	and thus
	\begin{align*}
		\dgm_j^\mathcal{R}&(X_1\times \cdots \times X_N,\dd_\mathsf{max})
		=\;\;\;
		\\ & \hspace{2cm}
		\left\{
		\left( \max_n a_n, \min_n b_n \right)  \; \Big\vert \;
		(a_n, b_n) \in \dgm^\mathcal{R}_{j_n}(X_n,\dd_n) \;\;,\;
		\sum\limits_{n=1}^N j_n = j\;
		\right\}
	\end{align*}
	where $\dd_{\mathsf{max}}(\xx, \xx') := \max\limits_{1\leq n \leq N} \dd_n(x_n, x'_n)$ is the maximum metric.
\end{prop}

These results combined yield the following:

\begin{lemma}\label{lemma:PersTorus}
	If $a < \sqrt{3}\vert\widehat{F}(\kk_N)\vert$, then
	\[
	(a,b) \in \dgm^\mathcal{R}_j(\T^N_{\widehat{F}}, \|\cdot\|_\infty)
	\;\;\;\; , \;\;\;\;
	1\leq j \leq N
	\]
	if and only if $a = 0$ and  $b = \sqrt{3}\vert\widehat{F}(\kk_n)\vert$
	for some $1\leq n \leq N$.
	
	Moreover, if $1\leq n \leq N$, and
	$ 1\leq  n_1 <\cdots < n_\ell  \leq N$
	is the longest  sequence (i.e., largest $\ell$) for which
	\[
	\vert\widehat{F}(\kk_n)\vert =  \vert\widehat{F}(\kk_{n_1})\vert = \cdots= \vert\widehat{F}(\kk_{n_\ell})\vert
	\]
	then  $\left(0, \sqrt{3}\vert\widehat{F}(\kk_n)\vert\right)$
	appears in $\dgm^\mathcal{R}_j(\T^N_{\widehat{F}}, \|\cdot\|_\infty)$ with
	multiplicity
	\[
	\mu_j(n) :=
	{n_1 -1  \choose j -1} + \cdots + {n_\ell -1   \choose j -1}.\]
\end{lemma}
\begin{proof}
	By Proposition \ref{prop :RipsKunneth},
	we have that
	$(a,b) \in \dgm_j^\mathcal{R}(\T^N_{\widehat{F}}, \|\cdot\|_\infty)$, $1\leq j \leq N$,
	if and only if there exist integers
	$0 \leq j_1,\ldots, j_N \leq j$ and
	\[
	(a_n,b_n) \in \dgm_{j_n}^\mathcal{R}(S^1_{\vert\widehat{F}(\kk_n)\vert}
	\,,\, \vert\cdot \vert\,)
	\;\;\; , \;\;\; 1\leq n \leq N
	\]
	so that $j_1 + \cdots + j_N = j$, $a = \max\{a_1,\ldots , a_N\}$ and
	$b = \min\{b_1,\ldots b_N\}$.
	Assume without loss of generality that $a = a_1$.
	
	If
	\[
	a_1 < \sqrt{3}\vert\widehat{F}(\kk_N)\vert \leq \sqrt{3}\vert\widehat{F}(\kk_1)\vert
	\]
	then Eq. (\ref{eqn: RipsPersCircle}) implies that
	$a_1 = 0 $ (hence $a_1 = \cdots = a_N = 0$) and therefore
	\[
	b_n =
	\left\{
	\begin{array}{ccl}
		\infty & \hbox{if} & j_n = 0 \\[.2cm]
		\sqrt{3}\vert\widehat{F}(\kk_n)\vert & \hbox{if} & j_n = 1
	\end{array}
	\right.
	\;\;\;,\;\;\; \mbox{for all} \;\;\;
	1\leq n \leq N.
	\] The first part of the lemma readily follows from this and $j \geq 1$.
	
	Let us now address the multiplicity computation.
	We will do so by counting the number of distinct copies of
	$ \left(0 , \sqrt{3} \vert\widehat{F}(\kk_n)\vert\right)$ contributed
	to $\dgm_j^\mathcal{R}(\T^N_{\widehat{F}}, \|\cdot\|_\infty)$
	by
	each index $1\leq n_1 < \cdots < n_\ell \leq N$.
	Indeed,  start with $n_1$ and assume
	$1 \leq j \leq n_1$.
	Then, each choice of $j-1$ integers
	$ 1 \leq m_1 < \cdots < m_{j-1} < n_{1}  $
	yields a set of indices
	\[
	\mathcal{M}(n_1) = \{m_1,\ldots, m_{j-1}, n_1\}
	\]
	parametrizing
	a unique way of writing $ \left(0 , \sqrt{3} \vert\widehat{F}(\kk_n)\vert\right)$ as
	$\left(\max\limits_m a_m \;,\; \min\limits_m b_m\right)$ for
	\[
	(a_m , b_m) =
	\left\{
	\begin{array}{ccl}
		\left(0, \sqrt{3}\vert\widehat{F}(\kk_m)\vert \right)& \hbox{if} & m \in \mathcal{M}(n_1) \\[.2cm]
		(0, \infty) & \hbox{if} & m \in \{1,\ldots ,N\} \smallsetminus \mathcal{M}(n_1).
	\end{array}
	\right.
	\]
	Since there are ${n_1 - 1 \choose j -1}$ ways of choosing $\mathcal{M}(n_1)$,  the sets
	$\mathcal{M}(n_1), \ldots, \mathcal{M}(n_\ell) $ are all distinct,
	and this computation
	accounts for all copies, then this completes the proof.
\end{proof}

\begin{corollary}\label{Corollary: SurjectiveTorus}
	If $0 < \delta \leq \epsilon < \sqrt{3}\vert\widehat{F}(\kk_N)\vert$ and $1\leq j \leq N$,
	then the homomorphism
	\[ \iota_* :
	H_j\left(R_\delta(\T^N_{\widehat{F}}, \|\cdot\|_\infty); \F\right)
	\longrightarrow
	H_j\left(R_\epsilon(\T^N_{\widehat{F}}, \|\cdot\|_\infty); \F\right)
	\]
	induced by the
	inclusion $\iota: R_\delta(\T^N_{\widehat{F}}, \|\cdot\|_\infty)  \hookrightarrow R_\epsilon(\T^N_{\widehat{F}}, \|\cdot\|_\infty) $,
	is surjective.
\end{corollary}

Let
\[
P : \C^{\alpha} \longrightarrow \C^N
\]
be the projection onto the first $N$-coordinates.
The first thing to note is that
$P(\X_{K,f}) \subset \T^N_{\widehat{F}}$ (see Eq. (\ref{eqn: big_XKf})),
and since
\[
\|P(\zz) - P(\zz') \|_\infty \;\leq \; \|P(\zz) - P(\zz')\|_2 \;\leq\;  \|\zz - \zz'\|_2
\]
for every $\zz,\zz' \in \C^\alpha$,
then $P$ induces
simplicial maps at the level of Rips complexes
\[
\begin{array}{rccl}
	R_\epsilon(P) : &
	R_\epsilon(\X_{K,f} , \|\cdot \|_2)& \longrightarrow &R_{\epsilon}(\T^N_{\widehat{F}}, \|\cdot\|_\infty) \\
	& \sigma & \mapsto & P(\sigma)
\end{array}
\]
for every $\epsilon > 0$.
The idea now is to use $\mathcal{R}(P) = \{R_\epsilon(P)\mid \epsilon > 0\}$
in order to derive insights about the persistent homology of $\mathcal{R}(\X_{K,f}, \|\cdot\|_2)$
from that of $\mathcal{R}(\T^N_{\widehat{F}}, \|\cdot\|_\infty)$.

We have the following,
\begin{lemma}\label{lemma:SurjectiveP}
	For all $0 < \epsilon  < \sqrt{3}\vert\widehat{F}(\kk_N)\vert$ and $1\leq j \leq N $, the homomorphism
	\[
	R_\epsilon(P)_* : H_j(R_\epsilon(\X_{K,f}, \|\cdot\|_2) ;\F) \longrightarrow H_j(R_\epsilon(\T^N_{\widehat{F}},\|\cdot\|_\infty) ;\F)
	\]
	is surjective.
\end{lemma}
\begin{proof}
	The case $N=1$ is essentially Theorem 6.8 in \cite{sw1pers},
	so assume $N\geq 2$.
	
	Our first claim is that  the projection  $P : (\C^\alpha, \|\cdot \|_2) \longrightarrow (\C^N, \|\cdot\|_\infty)$
	restricts to a homeomorphism
	\[
	P\,:\,
	\,\overline{\X_{K,f}}\,
	\xrightarrow{\hspace{.25cm}\cong\hspace{.25cm}}
	\,\T^N_{\widehat{F}}
	\]
	Indeed, surjectivity is   the content of Kronecker's Theorem (\ref{Kroneckermulti}),
	so in order to check injectivity,
	assume that   $\xx, \xx' \in \X_{K,f}$ have $P(\xx) = P(\xx')$, and write
	\[
	\xx =
	\begin{bmatrix}
		\widehat{F}(\kk_1)e^{i\<\kk_1,\omega\>t} \\
		\vdots \\
		\widehat{F}(\kk_\alpha)e^{i\<\kk_\alpha,\omega\>t}
	\end{bmatrix}
	\;\;\; , \;\;\;
	\xx' =
	\begin{bmatrix}
		\widehat{F}(\kk_1)e^{i\<\kk_1,\omega\>t'} \\
		\vdots \\
		\widehat{F}(\kk_\alpha)e^{i\<\kk_\alpha,\omega\>t'}
	\end{bmatrix}
	\]
	for $t,t' \in \R$.
	Since $P(\xx)= P(\xx')$ and $\widehat{F}(\kk_r) \neq 0$ for $r=1,\ldots, \alpha$,
	then there exist $m_1,m_2 \in \Z$ for which
	\begin{eqnarray*}
		\<\kk_1,\omega\>(t - t') &=& 2\pi m_1 \\
		\<\kk_2, \omega\>(t - t') &=& 2\pi m_2
	\end{eqnarray*}
	If $t\neq t'$, then we would have that
	\[
	\<m_2\kk_1 - m_1\kk_2 , \omega\> = 0
	\]
	contradicting either the incommensurability of $\omega $, or the $\Q$-linear independence of the vectors $\kk_1,\kk_2,\ldots, \kk_N$.
	Thus $t = t'$, showing that $P$ is injective on $\X_{K,f}$, and continuity plus Hausdorffness improves this to injectivity on  $\overline{\X_{K,f}}$.
	Finally, since $P$ provides a continuous bijection between $\overline{\X_{K,f}}$ and $\T^N_{\widehat{F}}$,
	and the former is compact (since it is closed and bounded), then $P$ yields the desired homeomorphism.
	
	Now, given $\epsilon > 0$,  let $ 0 < \delta_\epsilon \leq \epsilon$ be so that
	\begin{equation}\label{eqn: UniformCont}
		\|P(\xx) - P(\xx')\|_\infty < \delta_\epsilon \;\;\; \mbox{ always implies } \;\;\; \|\xx - \xx'\|_2 < \epsilon
	\end{equation}
	for $\xx,\xx' \in \X_{K,f}$.
	The existence of $\delta_\epsilon> 0$ follows from the uniform continuity of
	$P^{-1} : \T^N_{\widehat{F}} \longrightarrow \overline{\X_{K,f}}$,
	and replacing $\delta_\epsilon$ with $\min\{\delta_\epsilon,\epsilon\}$ if necessary.
	Density of $\X_{K,f}$ in $\overline{\X_{K,f}}$ implies that
	for each $\zz \in \T^N_{\widehat{F}}$ there is $\xx_\zz \in \X_{K,f}$  so that
	$\|P(\xx_\zz) - \zz\|_\infty < \frac{\delta_\epsilon}{4}$.
	Fixing a choice of  $\xx_\zz$ for each $\zz$ defines a function
	\[
	\begin{array}{rccl}
		\nu: & \T^N_{\widehat{F}} & \longrightarrow  & \X_{K,f} \\
		& \zz & \mapsto & \xx_\zz
	\end{array}
	\]
	satisfying
	\[
	\|P\circ \nu (\zz) - \zz\|_\infty  < \frac{\delta_\epsilon}{4}
	\]
	for every $\zz \in \T^N_{\widehat{F}}$.
	Therefore, if $\zz,\zz' \in \T^N_{\widehat{F}}$  are so that  $\|\zz - \zz'\|_\infty < \frac{\delta_\epsilon}{2}$,
	then
	\begin{eqnarray*}
		\|P\circ\nu(\zz) - P\circ\nu(\zz')\|_\infty &\leq & \|P\circ \nu(\zz) - \zz\|_\infty + \|\zz - \zz'\|_\infty + \|\zz' - P\circ\nu(\zz')\|_\infty \\[.2cm]
		&<& \frac{\delta_\epsilon}{4} + \frac{\delta_\epsilon}{2} + \frac{\delta_\epsilon}{4} \\[.2cm]
		&=& \delta_\epsilon
	\end{eqnarray*}
	which implies
	$
	\|\nu(\zz) - \nu(\zz')\|_2
	< \epsilon
	$ (by Eq. (\ref{eqn: UniformCont})),
	and   $\nu$ extends to a simplicial map
	\[
	\begin{array}{rccl}
		R(\nu) : &R_{\frac{\delta_\epsilon}{2}}(\T^N_{\widehat{F}}, \|\cdot\|_\infty)&
		\longrightarrow &
		R_\epsilon(\X_{K,f}, \|\cdot\|_2) \\
		&\sigma &\mapsto & \nu(\sigma)
	\end{array}
	\]
	at the level of Rips complexes.
	
	We claim that
	$R_\epsilon(P)\circ R(\nu)$
	is contiguous to the inclusion
	\[
	\iota :
	R_{\frac{\delta_\epsilon}{2}}(\T^N_{\widehat{F}}, \|\cdot\|_\infty)
	\hookrightarrow
	R_{\epsilon}(\T^N_{\widehat{F}}, \|\cdot\|_\infty)
	\]
	Indeed,
	if  $\zz,\zz' \in \T^N_{\widehat{F}}$ are so that  $\|\zz - \zz'\|_\infty < \frac{\delta_\epsilon}{2}$,
	then
	\begin{eqnarray*}
		\|\zz'  - P\circ \nu(\zz)\|_\infty &\leq &  \|\zz' - \zz\|_\infty + \|P\circ \nu (\zz) - \zz\|_\infty \\[.2cm]
		&< & \frac{\delta_\epsilon}{2} + \frac{\delta_\epsilon}{4} \\[.2cm]
		& <  & \delta_\epsilon \\[.2cm]
		&\leq &\epsilon
	\end{eqnarray*}
	showing that the set-theoretic union
	\[
	\iota(\sigma)
	\;\cup \;
	\big(R_\epsilon(P) \circ R(\nu)\big)(\sigma)
	\]
	is an element of $R_\epsilon(\T^N_{\widehat{F}}, \|\cdot \|_\infty)$
	for every $\sigma \in R_{\frac{\delta_\epsilon}{2}}(\T^N_{\widehat{F}}, \|\cdot\|_\infty)$.
	
	Contiguity at the level of simplicial maps implies that $\iota_* = R_\epsilon(P)_* \circ R(\nu)_*$ in  homology,
	and since $\iota_*$ is surjective (Corollary \ref{Corollary: SurjectiveTorus}),
	then it follows that $R_\epsilon(P)_*$ is also surjective.
\end{proof}

The next thing to note is that
Lemma \ref{lemma:PersTorus} together with
Lemma \ref{lemma:SurjectiveP}
yields the following
estimate for the number of toroidal persistent features  in $\mathcal{R}(\X_{K,f}, \|\cdot\|_2)$:
\begin{theorem}\label{theorem: longbarsX}
	Fix $1\leq n \leq N$, and let $1\leq n_1 < \cdots < n_\ell \leq N$ be the longest sequence for which
	\[
	\vert\widehat{F}(\kk_n)\vert =  \vert\widehat{F}(\kk_{n_1})\vert = \cdots = \vert\widehat{F}(\kk_{n_\ell})\vert.
	\]
	Then, for each $1\leq j \leq N$, the multiset cardinality of
	\[
	\left\{
	(0,b) \in \dgm^\mathcal{R}_j\left( \X_{K,f}\right) \; \Big\vert \; b \geq \sqrt{3}\vert\widehat{F}(\kk_n)\vert
	\right\}
	\]
	is  greater than or equal to
	$\mu_j(1) + \cdots + \mu_j(n)$, for
	\[
	\mu_j(n) =
	{n_1 - 1\choose j - 1} + \cdots + {n_\ell - 1 \choose j-1}.
	\]
\end{theorem}

In order to make statements on
$\dgm_j^\mathcal{R}(\SW_{d,\tau} S_K f)$, we will
leverage  the diagram

\begin{equation}
	\begin{tikzcd}
		\left(\SW_{d,\tau}S_Kf, \| \cdot \|_2 \right)\arrow[r,  bend left,"\Omega^{+}_{K,f}"]
		&
		\left(\X_{K,f}, \| \cdot \|_2\right)  \arrow[l, bend left, "\Omega_{K,f}"]
	\end{tikzcd}
\end{equation}

and the estimates in Rips persistence that it implies.
Here $\Omega_{K,f}$ is the Vandermonde matrix defined in Eq. (\ref{eqn: Omega}),
and $\Omega^+_{K,f}$ is its Moore-Penrose pseudoinverse (see \cite[III.3.4]{ben2003generalized}).

Let $0 < \sigma_{\mathsf{min}} \leq \sigma_{\mathsf{max}}$ be the smallest and  largest singular
values  of   $\Omega_{K,f}$, respectively.
Standard singular value decomposition arguments show that
\[
\left\|\Omega_{K,f} \uu \right\|_2
\leq
\sigma_{\mathsf{max}}\|\uu\|_2
\;\;\;\;\;\;
\mbox{and}
\;\;\;\;\;\;
\left\|\Omega^+_{K,f}\vv\right\|_2
\leq
\frac{1}{\sigma_{\mathsf{min}}}\|\vv\|_2
\]
for every $\uu\in \C^\alpha$ and $\vv \in \C^{d+1}$,
and thus we have  induced simplicial maps
\[
\Omega :
R_\epsilon(\X_{K,f}) \longrightarrow
R_{\epsilon\sigma_{\mathsf{max}}}(\SW_{d,\tau} S_K f)
\]

\[
\Omega^+ :
R_{\delta\sigma_{\mathsf{min}}}(\SW_{d,\tau} S_K f) \longrightarrow
R_{\delta}(\X_{K,f})
\]
at the level of Rips complexes.
Let
\[
\kappa(\Omega_{K,f}) = \frac{\sigma_{\mathsf{max}}}{\sigma_{\mathsf{min}}}\]
denote the condition number of $\Omega_{K,f}$.
Then
for every $\epsilon \leq \epsilon'$,
the    diagram

\begin{equation}\label{eqn: diagramRipsOmega}
	\begin{tikzcd}
		& R_{\epsilon\sigma_{\mathsf{min}}}(\SW_{d,\tau} S_K f)\arrow[r, hook]\arrow[d, "\Omega^+"]
		& R_{\epsilon'\sigma_{\mathsf{min}}}(\SW_{d,\tau} S_K f) \arrow[d, "\Omega^+"]
		\\
		R_{\epsilon/\kappa\left(\Omega_{K,f}\right)}(\X_{K,f})\arrow[ru, "\Omega"] \arrow[r, hook]
		& R_\epsilon(\X_{K,f}) \arrow[r, hook]
		& R_{\epsilon'}(\X_{K,f})
	\end{tikzcd}
\end{equation}

commutes.
The horizontal maps are inclusions, and commutativity
follows from noting  that $\Omega^+_{K,f}$ is a left inverse
of $\Omega_{K,f}$.
Indeed,  the latter (tall skinny) matrix is  full-rank
with our choice of $d$ and $\tau$.
Taking homology in dimension $j\in \N$ we get the induced homomorphisms
\begin{equation*}
	\begin{tikzcd}
		&H_j^{\epsilon\sigma_{\mathsf{min}},\; \epsilon'\sigma_{\mathsf{min}}}(\mathcal{R}(\SW_{d,\tau} S_K f) ; \F) \arrow[d,"\Omega^+_*"]
		\\
		H_j^{\epsilon/\kappa\left(\Omega_{K,f}\right),\; \epsilon'}(\mathcal{R}(\X_{K,f});\F) \arrow[r,hook]
		& H_j^{\epsilon,\, \epsilon'}(\mathcal{R}(\X_{K,f});\F)
	\end{tikzcd}
\end{equation*}
at the level of persistent homology groups (See Eq. (\ref{eqn: PersHomoGroup})),
where the horizontal map is an inclusion as linear spaces.
Commutativity of the diagram in Eq. (\ref{eqn: diagramRipsOmega})
implies that
\[
H_j^{\epsilon/\kappa\left(\Omega_{K,f}\right),\; \epsilon'}(\mathcal{R}(\X_{K,f});\F)
\subset
\mathsf{Img}\left(\Omega_*^+\right)
\]
which, after taking dimensions, yields the following inequality of persistent Betti numbers (see Eq. (\ref{eqn: PerBettiNum})):
\[
\beta_j^{\epsilon\sigma_{\mathsf{min}},\; \epsilon'\sigma_{\mathsf{min}}}(\mathcal{R}(\SW_{d,\tau} S_K f))
\;\;\geq\;\;
\mathsf{rank}\left(\Omega_*^+\right)
\;\;\geq \;\;
\beta_j^{\epsilon/\kappa\left(\Omega_{K,f}\right),\; \epsilon'}(\mathcal{R}(\X_{K,f}))
\]

Letting $\epsilon \rightarrow 0$ and using Theorem
\ref{theorem: longbarsX}, we get the following:
\begin{theorem}\label{theorem: longbarsSW}
	Let $f(t) = F(t\omega) $ be quasiperiodic with frequency vector $\omega \in \R^N$
	and parent function $F\in C^r(\T^N)$, $r > \frac{N}{2}$.
	Fix parameters $K,d,\tau$ as before.
	
	Let $\sigma_{\mathsf{min}}>0 $ be the smallest singular value of the Vandermonde
	matrix $\Omega_{K,f}$ (see Eq. (\ref{eqn: Omega})), and
	for  $1\leq n \leq N$,  let $1\leq n_1 < \cdots < n_\ell \leq N$ be the longest sequence for which
	\[
	\vert\widehat{F}(\kk_n)\vert =  \vert\widehat{F}(\kk_{n_1})\vert = \cdots = \vert\widehat{F}(\kk_{n_\ell})\vert.
	\]
	Then, for each $1\leq j \leq N$,
	the multiset cardinality of
	\[
	\left\{
	(0,b) \in \dgm^\mathcal{R}_j\left( \SW_{d,\tau} S_K f \right)
	\;\; \Big\vert \;\;
	b \geq \sqrt{3}\vert\widehat{F}(\kk_n)\vert\sigma_{\mathsf{min}}
	\right\}
	\]
	is  greater than or equal to $\mu_j(1) + \cdots + \mu_j(n)$ for
	\[
	\mu_j(n)=
	{n_1 - 1\choose j - 1} + \cdots + {n_\ell - 1 \choose j-1}.
	\]
\end{theorem}

The Stability Theorem  for Rips persistence (Eq. (\ref{eqn: ThmStability})),
together with
Theorem \ref{theorem: longbarsSW}  and Corollary \ref{Corollary: DGMapproximation} yield the main result of this section.
\begin{theorem}\label{thm: persistenceslidingwindowf}
	With the same hypotheses of Theorem \ref{theorem: longbarsSW},
	and for $1\leq j,n \leq N$,
	the multiset cardinality of
	\[
	\left\{
	(a,b) \in \dgm^\mathcal{R}_j\left( \SW_{d,\tau} f \right)
	\;\; \Big\vert \;\;
	b-a \;\geq\; \sqrt{3}\vert\widehat{F}(\kk_n)\vert\sigma_{\mathsf{min}}
	- 4\sqrt{d+1}\|f - S_K f\|_\infty
	\right\}
	\]
	is greater than or equal to $\mu_j(1) + \cdots +
	\mu_j(n)$.
\end{theorem}

The extremal singular values of Vandermonde matrices with nodes
in the unit circle
haven been extensively studied   in the harmonic analysis
and computational mathematics literature \cite{aubel2019vandermonde, moitra2015super, ferreira1999super}.
In particular, the lower bound on $\sigma_{\mathsf{min}}$ from \cite[Eq. (55)]{aubel2019vandermonde} implies the following.

\begin{corollary}\label{cor: persistencedelta}
	With the same hypotheses of Theorem \ref{theorem: longbarsSW}, and if
	\[
	d \;>\;  \frac{1}{\delta_\omega}- \frac{3}{2}
	\;\;\;\;\;\; \;\;,\;\;\; \;\;\;
	\delta_\omega :=
	\min_{1\leq   \ell < m \leq \alpha} \,
	\frac{1}{\pi}\arcsin\left(
	\frac{\left\vert e^{i\<\kk_\ell,\omega\>\tau} - e^{i\<\kk_m,\omega\>\tau}\right\vert}
	{2}
	\right)
	\]
	then, for each $1\leq j,n \leq N$,
	the multiset cardinality of
	\[
	\left\{
	(a,b) \in \dgm^\mathcal{R}_j\left( \SW_{d,\tau}f \right)
	\Big\vert
	b-a \geq \sqrt{3}\vert\widehat{F}(\kk_n)\vert\sqrt{d + \frac{3}{2} - \frac{1}{\delta_\omega}}
	- 4\sqrt{d+1}\|f - S_K f\|_\infty
	\right\}
	\]
	is always greater than or equal to $\mu_j(1) + \cdots + \mu_j(n)$.
\end{corollary}

This brings us to the end of the theoretical quasiperiodicity analysis in this paper. In the next section, we focus on examples and applications.

\section{Experiments and Applications}\label{section: applications}
This section has two goals: first, to illustrate the pipeline developed in this paper for the analysis of quasiperiodic time series data.
Indeed, we will utilize a synthetic example to review the optimization of $d$ and $\tau$, evaluate our theoretical lower bounds on persistence, and study the effects of noise on sliding window persistence.
The second goal is to provide an example of how quasiperiodicity
can arise in naturally-occurring time series data.
Specifically, we will study a sound recording  of \emph{dissonance},
and illustrate how quasiperiodicity emerges  through the lens of sliding windows and persistence.

\subsection{Computational pipeline and valuation of theoretical lower bounds}\label{subsection: example3freq}
Let
\[
f(t) = 2\sin(t) + 1.8\sin\left(\sqrt{3}t\right)
\;\;\;\mbox{ for }\;\;\;
0 \leq t \leq 60\pi
\]
with graph shown
in Figure \ref{fig: Exp1} below.

\begin{figure}[!htb]
	\centering
	\includegraphics[width=0.8\textwidth]{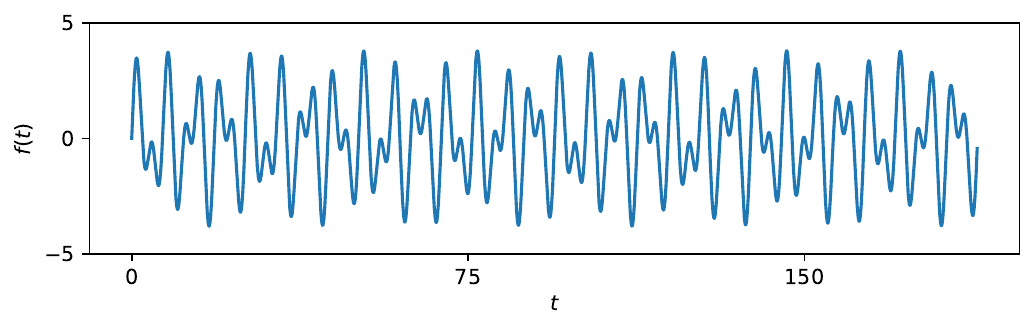}
	\caption{$f(t) = 2\sin(t) + 1.8\sin\left(\sqrt{3}t\right) $ for $0\leq t\leq 60\pi$}
	\label{fig: Exp1}
\end{figure}
We sample $f$  at 10,000 evenly spaced points;
that is, at each
\[
t\in T = \left\{ \frac{3n\pi}{500} \mid n = 0,\ldots, \,9,999\right\}
\]
producing a discrete time series for which
the Discrete Fourier Transform is computed (see Figure \ref{fig: Exp2} below).
We note that since $f$ is real-valued, then $\vert\widehat{f}(\xi)\vert$
is symmetric with respect to the origin,
$\omega = \left(1,\sqrt{3}\right)$, $\kk_1 = (1,0)$, $\kk_2 = (0,1)$, $\kk_3 = (-1,0)$, $\kk_4 = (0, -1)$, and that
$\vert\widehat{F}(\kk_1)\vert = \vert\widehat{F}(\kk_3)\vert \approx 1$,
$\vert\widehat{F}(\kk_2)\vert = \vert\widehat{F}(\kk_4)\vert \approx 0.9$.

\begin{figure}[!htb]
	\centering
	\includegraphics[width=0.85\textwidth]{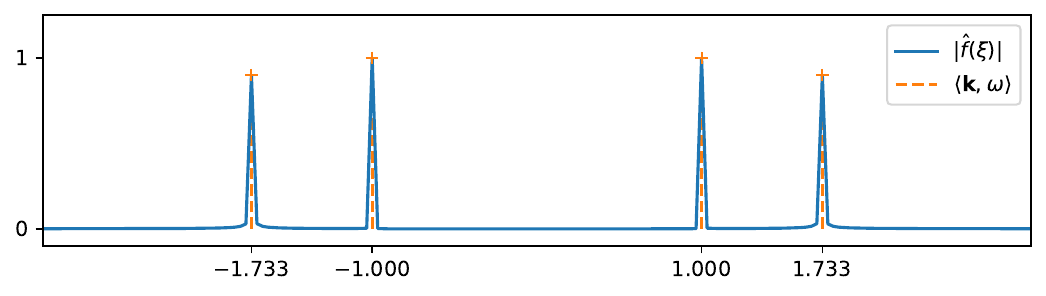}
	\caption{Modulus of the DFT for $f(T)$}
	\label{fig: Exp2}
\end{figure}

The number of prominent peaks in $\vert\widehat{f}(\xi)\vert$ is used---as described in Section \ref{subSection:d}---to select $d= 4$,
while the peak locations $\<\kk,\omega\>$  define the function $\Gamma(x)$ (see Eq. (\ref{eqn: Gopt})) whose minimizer over $\left[0, \tau_{\mathsf{max}}\right]$, for $\tau_\mathsf{max} = \frac{3}{4}\frac{60\pi}{d}$,  yields the choice $\tau \approx 11.9577$ as described in Section \ref{subSection:tau}.
The value of $\tau_{\mathsf{max}}$ is selected to guarantee that
the window size $d\tau$ is less  than that of the domain $T$ over which $f$ is evaluated.

The number of points in $T$ is already large enough that computing the Rips persistent
homology of $SW_{d,\tau} f(T)$, using standard software \cite{ripser},
is  algorithmically intensive.
Thus, we take a \texttt{maxmin} subsample
$ SW_{d,\tau} f(\widetilde{T})$ (see Example \ref{example: tau})
by selecting $\widetilde{T} \subset T$ with 1,000 points,
and compute
$\dgm_j^\mathcal{R}\left(SW_{d,\tau} f(\widetilde{T})\right)$
in dimensions $j=1,2$ and coefficients in $\F = \Z_2$.

Since $\widetilde{T} \neq \R$, then the lower bounds on persistence from Theorem \ref{theorem: longbarsSW} do not readily apply to the diagrams $\dgm_j^\mathcal{R}(SW_{d,\tau} f(\widetilde{T}))$.
That said, the stability theorem implies that the inequality
can be corrected to
\begin{equation}\label{eqn: BoundsHausCorrect}
	b-a \;\geq\; \sqrt{3}\vert\widehat{F}(\kk_n)\vert\sigma_{\mathsf{min}}
	- 4\sqrt{d+1}\|f - S_K f\|_\infty -
	4d_H\left(SW_{d,\tau} f (\widetilde{T}),
	\SW_{d,\tau} f\right)
\end{equation}
where, for this example, the Hausdorff distance term was estimated
as
\[
d_H\left(SW_{d,\tau} f(\widetilde{T}), \SW_{d,\tau} f\right)
\approx 0.54292.
\]
Figure \ref{fig:BoundsHaussNoHauss} below shows
the Rips persistence diagrams $\dgm_j^\mathcal{R}(SW_{d,\tau} f(\widetilde{T}))$, as well as the estimated lower bounds in persistence
with and without the correction term on Hausdorff distance.

\begin{figure}[!htb]
	\centering
	\includegraphics[width=0.95\textwidth]{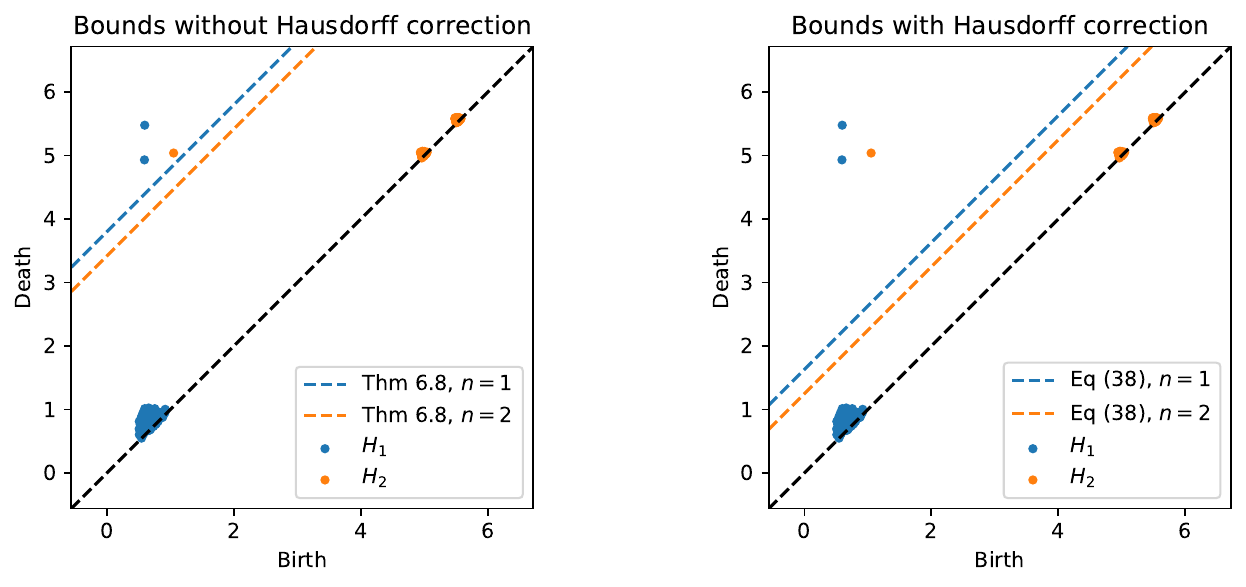}
	\caption{Rips persistence diagrams of
		$SW_{d,\tau} f(\widetilde{T})$ in dimensions $j=1$ (blue) and $j=2$ (orange), and coefficients in $\F = \Z_2$.
		Lower bounds on persistence are shown with dashed lines.
		Left: no Hausdorff correction, and  Right: Hausdorff correction}
	\label{fig:BoundsHaussNoHauss}
\end{figure}

Next, we aim to illustrate the effect of introducing noise to a quasiperiodic signal by examining the sliding window persistence of the resulting signal.
Note that in \cite[Section 4.1]{tralie2018quasi}, the authors extensively studied the effect of adding different types and levels of noise to recurrent videos. They measured the accuracy of a binary classification task inspired by their persistence based (quasi)periodicity scores and showed that persistence separates recurrent and non-recurrent videos under noise very well.
Here, we use the function $f(t) = 2\sin(t) + 1.8\sin\left(\sqrt{3}t\right)$ at $t\in T = \left\{ \frac{3n\pi}{500} \mid n = 0,\ldots, \,9,999\right\}$ and add random Gaussian noise to $f$. Then we use the Discrete Fourier Transform to determine the frequencies. For parameter selection, we choose $d$ based on the number of prominent peaks and compute the optimal $\tau$ as described in Section \ref{subSection:tau}. For each noise level, we compute the sliding window persistence for $800$ landmarks chosen via the \verb*|maxmin| subsampling process. 
In Figure \ref{fig:NoisePers} (Top), we track the maximum persistence (blue) and the second maximum persistence (purple) in dimension $1$ as we increase the Noise-to-Signal Ratio (NSR) defined as 
\[
\mathrm{NSR} = \sqrt{\frac{E[N^2]}{E[S^2]}}
\]
where $N$ is the Gaussian noise, $S$ is the signal $f$, and $E[\ \cdot \ ]$ is the expected value. We also show, for contrast, all other lower persistences (gray), i.e. third maximum persistence, fourth maximum persistence, and so on.
Similarly, in Figure \ref{fig:NoisePers} (Bottom), we track the maximum persistence (orange) and all other lower persistences (gray), i.e. second maximum persistence, third maximum persistence, etc., in dimension $2$.
\begin{figure}[!htb]
	\centering
	\begin{subfigure}{\textwidth}
		\centering
		\includegraphics[width = 0.8\textwidth]{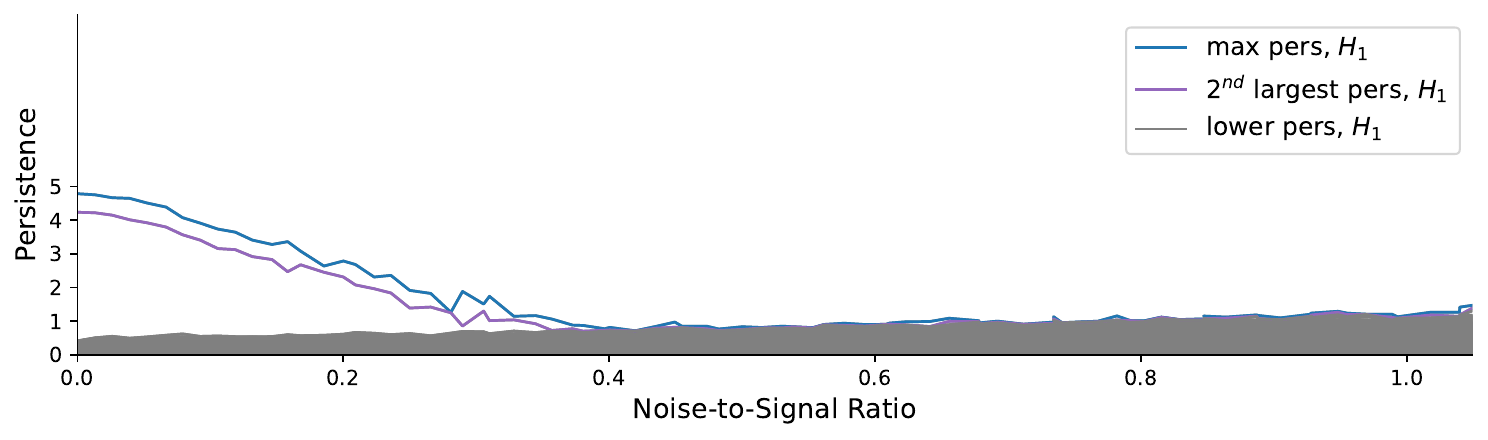}
	\end{subfigure}
	\begin{subfigure}{\textwidth}
		\centering
		\includegraphics[width = 0.8\textwidth]{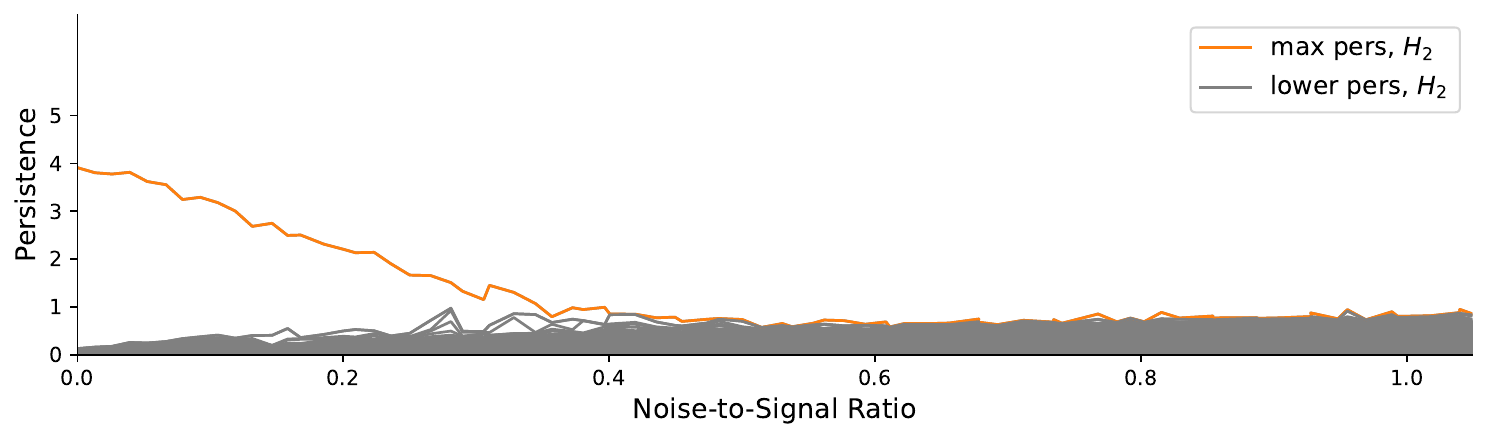}
	\end{subfigure}
	\caption
	{Sliding Window Persistence versus Noise-to-Signal Ratio. The top figure shows the two maximum persistences in dimension $1$ and the bottom figure shows the maximum persistence in dimension $2$. In both cases, the curves for lower persistences, i.e. all other persistences are also added for contrast.}
	\label{fig:NoisePers}
\end{figure}

\subsection{Application: Dissonance Detection in Music}
In music theory, consonance and dissonance are classifications of multiple simultaneous tones. While the former is associated with pleasantness, the latter creates tension as experienced by the listener.
Perfect dissonance occurs when the audio frequencies are irrational with respect to each other. One   such instance is the  \emph{tritone}, which is a musical interval that is halfway between two octaves. Mathematically, for a base frequency $\omega_1$, its tritone is $\sqrt{2}\omega_1$.
We will use the theory of sliding window embedding to quantify quasiperiodicity from a dissonant sample.
For the purpose of this application, we use a $5$-second audio recording of a brass horn playing the tritone\footnote{generously provided by Adam Huston.}. The signal was read using \textsf{wavfile.read()} and the resulting amplitude plot is shown in Figure \ref{fig:MusicAudioWave} (Top).
Like before, in order to perform sliding window analysis, we need to choose appropriate parameters $d$ and $\tau$.
We proceed exactly as before  with the spectral analysis shown in Figure
\ref{fig:MusicAudioWave} (Bottom).
\begin{figure}[!htb]
	\centering
	\subfloat{\includegraphics[width = \textwidth]{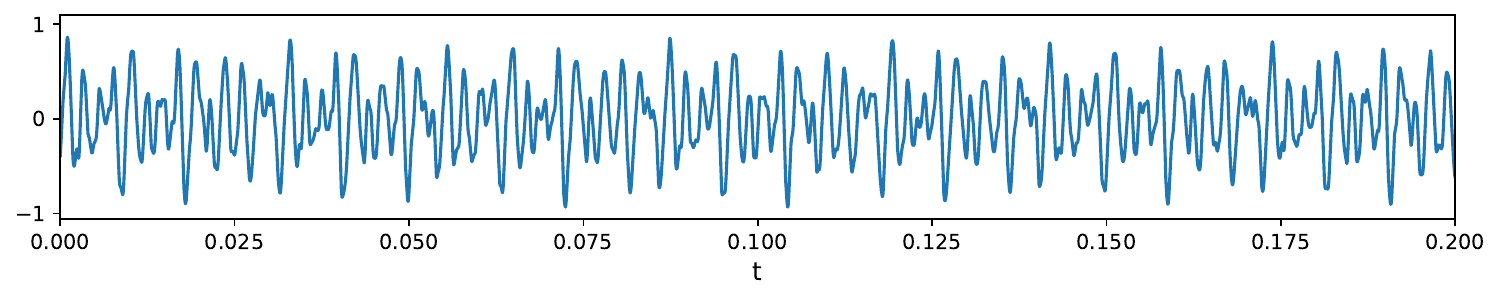}
	}
	
	\subfloat{\includegraphics[width = \textwidth]{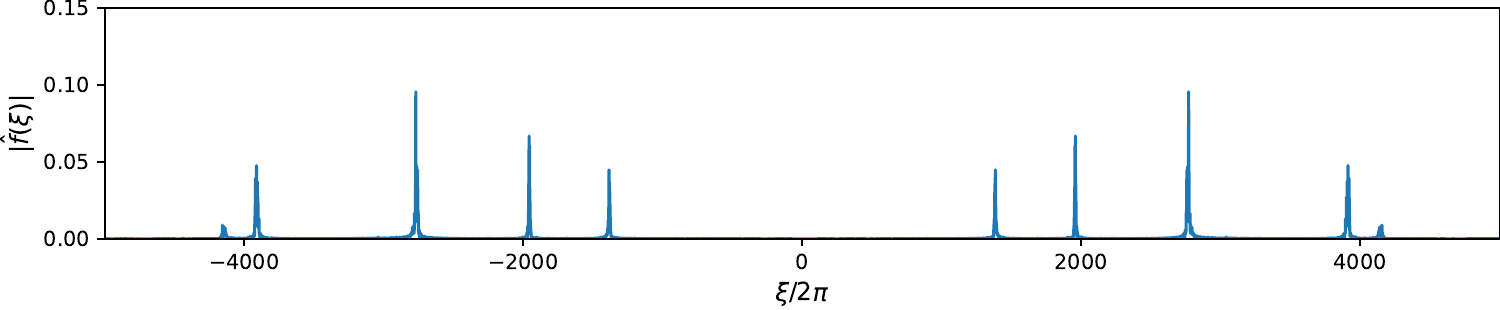}
	}
	\caption
	{(Left) The plot of a dissonant music sample created on a brass horn. (Right) The amplitude-frequency spectrum of the music sample plot using Fast Fourier Transform.}
	\label{fig:MusicAudioWave}
\end{figure}

We then find peaks with height at least $0.04$ and at least $100$ radians per second apart to detect prominent frequencies which we will use for estimation of the embedding parameters. See Table \ref{tab:musicfreq}.
\begin{table}[!htb]
	\centering
	\begin{tabular}{|c|c|c|c|c|}
		\hline
		\textbf{Angular Frequencies} & 1384.93 & 1957.83 & 2769.86& 3911.93 \\
		\hline
		\textbf{Frequencies (Hz)} & 220.41  & 311.59 & 440.83 & 622.60\\
		\hline
		\textbf{Proportion} & 1 & 1.4137 $\approx \sqrt{2}$   & 2 & 2.8246 $\approx2\sqrt{2}$ \\
		\hline
	\end{tabular}
	\caption
	{List of frequencies in the positive side of the (symmetric) amplitude-frequency spectrum: First row: list of detected frequencies. Second row: their conversion to Hertz. Third row: ratio with respect to the first row.}\label{tab:musicfreq}
\end{table}
The resulting embedding parameters are $d = 8$ and $\tau = 0.0285736$.
We use cubic splines to compute the sliding window vectors and present the PCA representation of the point cloud, along with the persistence diagrams computed for $1300$ landmarks, i.e. \verb"maxmin" subsample as defined in Example \ref{example: tau}, in Figure \ref{fig:MusicPersDiag}.
\begin{figure}[!htb]
	\centering
	\hspace{-1cm}
	\begin{subfigure}{0.4\textwidth}
		\centering
		\includegraphics[width = \textwidth]{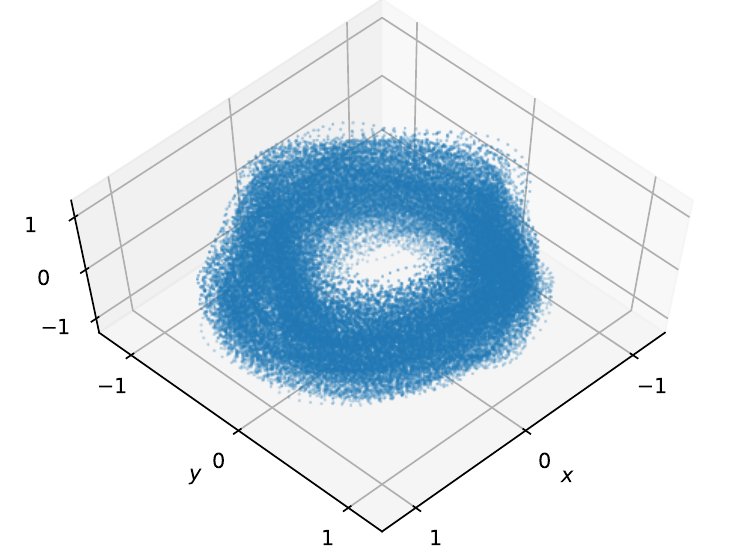}
	\end{subfigure}
	\hspace{.5cm}
	\begin{subfigure}{0.33\textwidth}
		\centering
		\includegraphics[width = \textwidth]{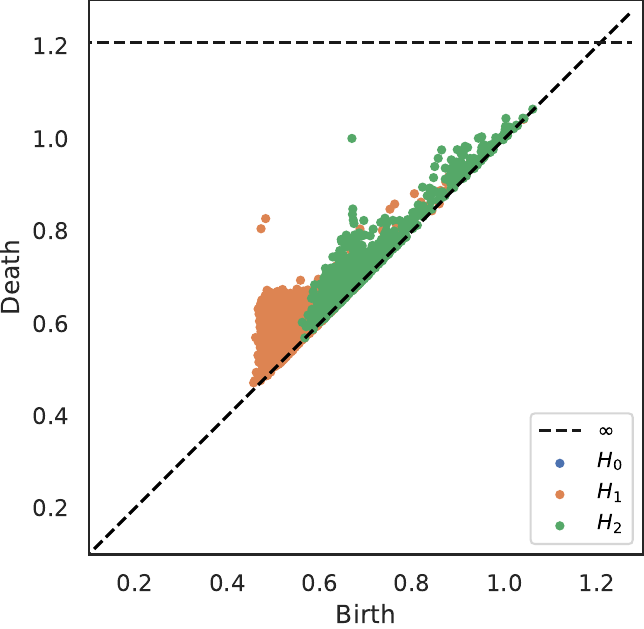}
	\end{subfigure}
	\hspace{.5cm}
	\begin{subfigure}{0.15\textwidth}
		\centering
		\includegraphics[width = \textwidth]{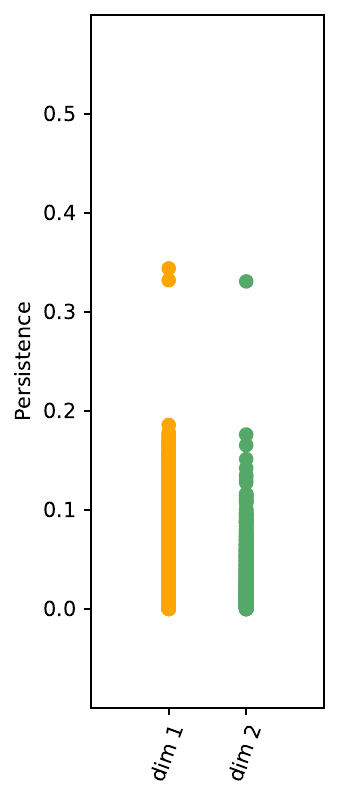}
	\end{subfigure}
	\caption
	{(Left) PCA representation of the sliding window point cloud. (Middle) Persistence Diagrams in homological dimensions $0$, $1$ and $2$. (Right) The persistence scatter plot.}
	\label{fig:MusicPersDiag}
\end{figure}

In Figure \ref{fig:MusicPersDiag}, the persistence diagrams (middle) indicate that the sliding window point cloud has two high persistence features in dimension $1$ and one high persistence feature in dimension $2$. This claim is validated with the persistence scatter plot (right). This tell us that the point cloud fills a two dimensional torus (perhaps a very twisted one) embedded in $\R^9$, which verifies that the dissonant music sample was indeed quasiperiodic.

\section*{Acknowledgments}

This work was partially supported by the National Science Foundation through grants DMS-1622301,  CCF-2006661,
and CAREER award  DMS-1943758. The authors of this paper would like to thank Adam Huston for the audio recording of the brass horn. The first author would like to thank Rosemarie Bongers for discussions on some of the Harmonic Analysis aspects of this paper. 

\section*{Declaration}

\subsection*{Conflict of interest}
On behalf of all authors, the corresponding author states that there is no conflict of interest.
	
\bibliographystyle{amsplain}
\bibliography{SWQuasiperiodicityBibliography}	

\providecommand{\bysame}{\leavevmode\hbox to3em{\hrulefill}\thinspace}
\providecommand{\MR}{\relax\ifhmode\unskip\space\fi MR }
\providecommand{\MRhref}[2]{%
  \href{http://www.ams.org/mathscinet-getitem?mr=#1}{#2}
}
\providecommand{\href}[2]{#2}
\begin{thebibliography}{10}

\bibitem{vrcircle}
Micha{\l} Adamaszek and Henry Adams, \emph{{The Vietoris--Rips complexes of a
  circle}}, Pacific Journal of Mathematics \textbf{290} (2017), no.~1, 1--40.

\bibitem{apostol2012modular}
Tom~M Apostol, \emph{{Modular functions and Dirichlet series in number
  theory}}, vol.~41, Springer Science \& Business Media, New York, 2012.

\bibitem{aubel2019vandermonde}
C{\'e}line Aubel and Helmut B{\"o}lcskei, \emph{Vandermonde matrices with nodes
  in the unit disk and the large sieve}, Applied and Computational Harmonic
  Analysis \textbf{47} (2019), no.~1, 53--86.

\bibitem{ripser}
Ulrich Bauer, \emph{Ripser}, URL: https://github. com/Ripser/ripser (2016).

\bibitem{belloy2017dynamic}
Micha{\"e}l Belloy, Maarten Naeyaert, Georgios Keliris, Anzar Abbas, Shella
  Keilholz, Annemie Van Der~Linden, and Marleen Verhoye, \emph{{Dynamic resting
  state fMRI in mice: detection of Quasi-Periodic Patterns}}, Proceeding of the
  International Soc. Magn. Reson. Med (2017), 0961.

\bibitem{ben2003generalized}
Adi Ben-Israel and Thomas~NE Greville, \emph{Generalized inverses: theory and
  applications}, vol.~15, Springer Science \& Business Media, New York, 2003.

\bibitem{broer2004kam}
Henk Broer, \emph{{KAM theory: the legacy of Kolmogorov’s 1954 paper}},
  Bulletin of the American Mathematical Society \textbf{41} (2004), no.~4,
  507--521.

\bibitem{chazal2016observable}
F~Chazal, WW~Crawley-Boevey, and V~de~Silva, \emph{The observable structure of
  persistence modules}, Homology, Homotopy and Applications \textbf{18} (2016),
  no.~2, 247--265.

\bibitem{chazal2014persistence}
Fr{\'e}d{\'e}ric Chazal, Vin De~Silva, and Steve Oudot, \emph{Persistence
  stability for geometric complexes}, Geometriae Dedicata \textbf{173} (2014),
  no.~1, 193--214.

\bibitem{crawley2015decomposition}
William Crawley-Boevey, \emph{Decomposition of pointwise finite-dimensional
  persistence modules}, Journal of Algebra and its Applications \textbf{14}
  (2015), no.~05, 1550066.

\bibitem{das2016measuring}
Suddhasattwa Das, Chris~B Dock, Yoshitaka Saiki, Martin Salgado-Flores, Evelyn
  Sander, Jin Wu, and James~A Yorke, \emph{Measuring quasiperiodicity}, EPL
  (Europhysics Letters) \textbf{114} (2016), no.~4, 40005.

\bibitem{emrani2014persistent}
Saba Emrani, Thanos Gentimis, and Hamid Krim, \emph{{Persistent homology of
  delay embeddings and its application to wheeze detection}}, IEEE Signal
  Processing Letters \textbf{21} (2014), no.~4, 459--463.

\bibitem{ferreira1999super}
P.J.S.G. Ferreira, \emph{{Super-resolution, the recovery of missing samples and
  Vandermonde matrices on the unit circle}}, Proceedings of the Workshop on
  Sampling Theory and Applications, Loen, Norway, 1999.

\bibitem{gakhar2020topological}
Hitesh Gakhar, \emph{A topological study of toroidal dynamics}, Ph.D. thesis,
  Michigan State University, 2020.

\bibitem{Kunneth}
Hitesh Gakhar and Jose~A. Perea, \emph{K\"unneth formulae in persistent
  homology}, arXiv preprint arXiv:1910.05656 (2019).

\bibitem{gomez2010collocation}
Gerard G{\'o}mez, Josep-Maria Mondelo, and Carles Sim{\'o}, \emph{{A
  collocation method for the numerical Fourier analysis of quasi-periodic
  functions. I: Numerical tests and examples}}, Discrete \& Continuous
  Dynamical Systems-B \textbf{14} (2010), no.~1, 41.

\bibitem{grafakos2008classical}
Loukas Grafakos, \emph{{Classical Fourier Analysis}}, vol.~2, Springer, New
  York, 2008.

\bibitem{hollander1992quasi}
A~Hollander and J~Van~Paradijs, \emph{Quasi-periodic oscillations in {TT
  Arietis}}, Astronomy and Astrophysics \textbf{265} (1992), 77--81.

\bibitem{khasawneh2018chatter}
Firas~A Khasawneh, Elizabeth Munch, and Jose~A. Perea, \emph{Chatter
  classification in turning using machine learning and topological data
  analysis}, IFAC-PapersOnLine \textbf{51} (2018), no.~14, 195--200.

\bibitem{laskar1993frequency}
Jacques Laskar, \emph{{Frequency analysis for multi-dimensional systems. Global
  dynamics and diffusion}}, Physica D: Nonlinear Phenomena \textbf{67} (1993),
  no.~1-3, 257--281.

\bibitem{moitra2015super}
Ankur Moitra, \emph{{Super-resolution, extremal functions and the condition
  number of Vandermonde matrices}}, Proceedings of the forty-seventh annual ACM
  symposium on Theory of computing, 2015, pp.~821--830.

\bibitem{morozov2005persistence}
Dmitriy Morozov, \emph{{Persistence algorithm takes cubic time in worst case}},
  BioGeometry News, Dept. Comput. Sci., Duke Univ \textbf{2} (2005).

\bibitem{toroidal}
Jose~A. Perea, \emph{Persistent homology of toroidal sliding window
  embeddings}, 2016 IEEE International Conference on Acoustics, Speech and
  Signal Processing (ICASSP), IEEE, 2016, pp.~6435--6439.

\bibitem{perea2019notices}
\bysame, \emph{{Topological Time Series Analysis}}, Notices of the American
  Mathematical Society \textbf{66} (2019), no.~5.

\bibitem{perea2015sw1pers}
Jose~A. Perea, Anastasia Deckard, Steve~B Haase, and John Harer,
  \emph{{SW1PerS: Sliding windows and 1-persistence scoring; discovering
  periodicity in gene expression time series data}}, BMC bioinformatics
  \textbf{16} (2015), no.~1, 257.

\bibitem{sw1pers}
Jose~A. Perea and John Harer, \emph{{Sliding windows and persistence: An
  application of topological methods to signal analysis}}, Foundations of
  Computational Mathematics \textbf{15} (2015), no.~3, 799--838.

\bibitem{pollack1982quasi}
James~B Pollack and Owen~B Toon, \emph{{Quasi-periodic climate changes on Mars:
  A review}}, Icarus \textbf{50} (1982), no.~2-3, 259--287.

\bibitem{radovanovic2010hubs}
Milos Radovanovic, Alexandros Nanopoulos, and Mirjana Ivanovic, \emph{{Hubs in
  space: Popular nearest neighbors in high-dimensional data}}, Journal of
  Machine Learning Research \textbf{11} (2010), no.~sept, 2487--2531.

\bibitem{robins1999towards}
Vanessa Robins, \emph{Towards computing homology from finite approximations},
  Topology proceedings, vol.~24, 1999, pp.~503--532.

\bibitem{samoilenko2012elements}
Anatolii~M Samoilenko, \emph{Elements of the mathematical theory of
  multi-frequency oscillations}, vol.~71, Springer Science \& Business Media,
  Dordrecht, 2012.

\bibitem{slater1967gaps}
Noel~B Slater, \emph{Gaps and steps for the sequence n$\theta$ mod 1},
  Mathematical Proceedings of the Cambridge Philosophical Society, vol.~63,
  Cambridge University Press, 1967, pp.~1115--1123.

\bibitem{sos1958distribution}
Vera~T S{\'o}s, \emph{On the distribution mod 1 of the sequence n$\alpha$},
  Ann. Univ. Sci. Budapest, E{\"o}tv{\"o}s Sect. Math \textbf{1} (1958),
  127--134.

\bibitem{stein2016introduction}
Elias~M Stein and Guido Weiss, \emph{{Introduction to Fourier Analysis on
  Euclidean Spaces (PMS-32)}}, vol.~32, Princeton university press, Princeton,
  2016.

\bibitem{stewart2015algebraic}
Ian Stewart and David Tall, \emph{Algebraic number theory and fermat's last
  theorem}, CRC Press, Boca Raton, 2015.

\bibitem{takens1981detecting}
Floris Takens, \emph{Detecting strange attractors in turbulence}, Dynamical
  systems and turbulence, Warwick 1980, Springer, Germany, 1981, pp.~366--381.

\bibitem{ctralie2018ripser}
Christopher Tralie, Nathaniel Saul, and Rann Bar-On, \emph{{Ripser.py}: A lean
  persistent homology library for python}, The Journal of Open Source Software
  \textbf{3} (2018), no.~29, 925.

\bibitem{tralie2018slowmotion}
Christopher~J Tralie and Matthew Berger, \emph{{Topological eulerian synthesis
  of slow motion periodic videos}}, 2018 25th IEEE International Conference on
  Image Processing (ICIP), IEEE, 2018, pp.~3573--3577.

\bibitem{tralie2018quasi}
Christopher~J Tralie and Jose~A. Perea, \emph{{(Quasi)-Periodicity
  quantification in video data, using topology}}, SIAM Journal on Imaging
  Sciences \textbf{11} (2018), no.~2, 1049--1077.

\bibitem{vela2002time}
Luz~Vianey Vela-Arevalo, \emph{{Time-frequency analysis based on wavelets for
  Hamiltonian systems}}, Ph.D. thesis, California Institute of Technology,
  2002.

\bibitem{webber1994dynamical}
Charles~L Webber~Jr and Joseph~P Zbilut, \emph{Dynamical assessment of
  physiological systems and states using recurrence plot strategies}, Journal
  of applied physiology \textbf{76} (1994), no.~2, 965--973.

\bibitem{weixing1993quasiperiodic}
Ding Weixing, Huang Wei, Wang Xiaodong, and CX~Yu, \emph{Quasiperiodic
  transition to chaos in a plasma}, Physical review letters \textbf{70} (1993),
  no.~2, 170.

\bibitem{mammalsubharmonics}
Inka Wilden, Hanspeter Herzel, Gustav Peters, and G{\"u}nter Tembrock,
  \emph{Subharmonics, biphonation, and deterministic chaos in mammal
  vocalization}, Bioacoustics \textbf{9} (1998), no.~3, 171--196.

\bibitem{wojewoda1993complex}
Jerzy Wojewoda, Tomasz Kapitaniak, Ronald Barron, and John Brindley,
  \emph{Complex behaviour of a quasiperiodically forced experimental system
  with dry friction}, Chaos, Solitons \& Fractals \textbf{3} (1993), no.~1,
  35--46.

\bibitem{xu2019twisty}
Boyan Xu, Christopher~J Tralie, Alice Antia, Michael Lin, and Jose~A Perea,
  \emph{{Twisty Takens: A geometric characterization of good observations on
  dense trajectories}}, Journal of Applied and Computational Topology
  \textbf{3} (2019), no.~4, 285--313.

\bibitem{zbilut2002recurrence}
Joseph~P Zbilut, Nitza Thomasson, and Charles~L Webber, \emph{Recurrence
  quantification analysis as a tool for nonlinear exploration of nonstationary
  cardiac signals}, Medical engineering \& physics \textbf{24} (2002), no.~1,
  53--60.

\end{thebibliography}
	
\end{document}